\pdfoutput=1
\documentclass[bj,preprint,noshowframe,numbers]{imsart}

\usepackage{amsmath,amssymb,amsthm,mathtools}
\usepackage{microtype}
\usepackage{tikz}
\usetikzlibrary{patterns,decorations.pathreplacing} \tikzset{box/.style={draw,line width=.35pt},ladder/.style={draw,line width=.35pt},  opath/.style={draw=blue!70!black,line width=.5pt},  meet/.style={circle,fill=red,inner sep=1.1pt},  cote/.style={<->,>=latex,line width=.3pt}}
\usepackage{hyperref}
\hypersetup{
  pdftitle={Critical oriented percolation enhanced along random supercritical horizontal lines},
  pdfauthor={Lucas Leproult, Maria Eulália Vares}
}

\startlocaldefs

\theoremstyle{plain}
\newtheorem{theorem}{Theorem}[section]
\newtheorem{proposition}[theorem]{Proposition}
\newtheorem{lemma}[theorem]{Lemma}
\newtheorem{corollary}[theorem]{Corollary}

\theoremstyle{definition}
\newtheorem{definition}[theorem]{Definition}
\newtheorem{remark}[theorem]{Remark}

\newcommand{\Z}{\mathbb{Z}}
\newcommand{\N}{\mathbb{N}}
\newcommand{\R}{\mathbb{R}}
\newcommand{\Ze}{\widetilde{\mathbb{Z}}^{2}_{+}}
\newcommand{\pc}{p_{c}}
\newcommand{\pG}{p_{G}}
\newcommand{\pB}{p_{B}}
\newcommand{\PP}{\mathbb{P}}
\newcommand{\EE}{\mathbb{E}}
\newcommand{\Pxi}{\mathbb{P}^{\xi}}
\newcommand{\Exi}{\mathbb{E}^{\xi}}
\newcommand{\ind}{\mathbf{1}}
\newcommand{\reach}{\zeta}

\endlocaldefs

\begin{document}

\markboth{Lucas Leproult and Maria Eulália Vares}{Critical oriented percolation enhanced along lines}
\thispagestyle{plain}
\begin{center}
{\LARGE Critical oriented percolation enhanced along\\ random supercritical horizontal lines\par}
\vspace{1.5em}
{\large Lucas Leproult\footnotemark[1]\qquad Maria Eulália Vares\footnotemark[2]\par}
\vspace{1em}
{\large October 13, 2026\par}
\end{center}
\footnotetext[1]{Institut Polytechnique de Paris, Université Paris-Saclay, ENSAE Paris, 5 avenue Henry Le Chatelier, 91120 Palaiseau, France; \texttt{lucas.leproult@ip-paris.fr}}
\footnotetext[2]{Instituto de Matemática, Universidade Federal do Rio de Janeiro, Av. Athos da Silveira Ramos 149, CEP 21941-909, Rio de Janeiro, RJ, Brazil; \texttt{eulalia@im.ufrj.br}}
\vspace{1em}

\begin{center}
\textbf{Abstract}
\end{center}
{\small\leftskip=0.075\textwidth \rightskip=0.075\textwidth
\noindent We consider the oriented site percolation model in a random environment introduced by
Kesten, Sidoravicius and Vares (2022), in which the horizontal lines of the lattice are
independently declared bad or good, sites being open with a probability that
depends only on the type of their line. We study the boundary case where the bad lines are
critical and the good ones supercritical, so that the model is critical oriented site percolation enhanced along a
random set of good lines. We prove that percolation occurs for any positive density of the good lines. An arbitrarily sparse set
of arbitrarily weakly enhanced lines suffices, the density of good lines and the
strength of the enhancement being decoupled. The proof adapts the renormalization scheme of Archer, Hartarsky, Kolesnik,
Olesker-Taylor, Schapira and Valesin (2026). We show that the edge of the process travels at a
positive speed, which yields crossing estimates for tilted boxes; during runs of bad
lines, we rely on the box-crossing property of Duminil-Copin, Tassion and Teixeira (2018); and
the renormalized model dominates the oriented percolation with geometric defects of
Hilário, Sá, Sanchis and Teixeira (2026).


\noindent\textbf{MSC2020 subject classifications:} 60K35; 60K37; 82B43.
\par}

\section{Introduction}
\label{sec:intro}

We consider the model of Kesten, Sidoravicius and
Vares~\cite{KestenSidoraviciusVares2022}, which we recall in their notation. On
the rotated and rescaled lattice $\Ze:=\{(x,n)\in\Z\times\Z_{+}: x+n \text{ is even}\}$,
with graph structure obtained by placing edges between any two sites of $\Ze$ at
Euclidean distance $\sqrt{2}$ from each other, this is the following oriented
(northwest--northeast) site percolation model: the \emph{lines}
$H_{i}:=\{(x,n)\in\Ze : n=i\}$ are declared to be bad or good with
probabilities $\delta$ and $1-\delta$ respectively, independently of each other.
Given the environment of lines, sites on good lines are open with probability
$\pG$, and sites on bad lines are open with probability $\pB$, independently of
each other.

On a suitable probability space $(\Omega,\mathcal{F},\PP)$, where
$\PP=\PP_{\pG,\pB,\delta}$, we consider a Bernoulli sequence $\xi=(\xi_{i})_{i\in\Z_{+}}$ with $\PP(\xi_{i}=1)=\delta$,
which determines whether $H_{i}$ is bad or good, and a family of $\{0,1\}$-valued occupation variables
$\eta=(\eta_{z})_{z\in\Ze}$ which are conditionally independent given $\xi$, with
$\PP(\eta_{z}=1\mid\xi)=\pB$ if $z\in H_{i}$ with
$\xi_{i}=1$, and $\PP(\eta_{z}=1\mid\xi)=\pG$ if
$z\in H_{i}$ with $\xi_{i}=0$. If $\eta_{z}=1$ the site $z$ is open, and
otherwise it is closed. The sequence $\xi$ is called the
\emph{environment}, and we write $\Pxi:=\PP(\,\cdot\mid\xi)$ for the conditional
law given the environment, with expectation $\Exi$. We call the family $\eta$ the \emph{occupation configuration}.

An \emph{open oriented path} on the graph $\Ze$ is a path along which the second
coordinate is strictly increasing and all of the sites are open except possibly the initial one. Since the initial site of an open path is not required to be open, being open
is a condition to reach a site, not to leave it. From now on, all paths are oriented, and an \emph{open path} always means an open
oriented path. The open
cluster $C_{z}$ of a site $z\in\Ze$ is the set of sites which can be reached
from $z$ by an open path. We write $C_{0}$ for the open cluster of $(0,0)$ and $\Theta(\pG,\pB,\delta)=\PP(C_{0}\ \text{is infinite})$,
and we say that percolation occurs if $\Theta(\pG,\pB,\delta)>0$.

Let $\pc$ denote the critical probability for the standard oriented site
percolation model on $\Z_{+}\times\Z_{+}$, which, rotated by $\pi/4$, is the
homogeneous version of our model. Kesten, Sidoravicius and
Vares~\cite{KestenSidoraviciusVares2022} proved that, for $\pG>\pc$ and
$\pB>0$, the model percolates as soon as the density $\delta$ of bad lines is
small enough, depending on $\pG$ and $\pB$. They also observe that a
Peierls-type argument gives some $\pB>0$ and $\delta>0$ for which
$\Theta(\pG,\pB,\delta)=0$ even if $\pG=1$. For such subcritical bad lines, a
positive density of good lines does not suffice, however large $\pG$. In a similar layered model on the three-dimensional hexagonal lattice, with
additional unoriented edges within the layers, de Lima, Ungaretti and
Vares~\cite{LimaUngarettiVares2024} prove that, when percolation within the
layers is subcritical, bad layers may prevent percolation even when all upward
edges from the good layers are open, whereas critical percolation within the
layers yields percolation even when all layers are bad.

\subsection{Main result}

Here, we consider the boundary case $\pB=\pc$, $\pG>\pc$,
in which the bad lines are critical and the good ones supercritical. The model is then an enhancement of
critical oriented site percolation: every site is open with probability at least
$\pc$, and the sites of the good lines, which have density $1-\delta$, are open with the larger probability $\pG$. We ask how few good lines suffice for this
enhancement to make the model percolate, even when $\pG$ is arbitrarily close to
$\pc$. In other words, we ask whether the density $1-\delta$ of good lines needed
for percolation has to depend on the strength $\pG-\pc$ of the enhancement.

Our main result is the following.

\begin{theorem}
\label{thm:main}
Let $\pB=\pc$ and $\pG>\pc$. Then
\[
  \Theta(\pG,\pc,\delta)>0
  \qquad\text{for every } \delta<1 .
\]
In fact, for every $\delta\in(0,1)$,
\begin{equation}
  \label{eq:quenched}
  \PP(C_{0}\ \text{is infinite}\mid\xi)>0
  \qquad\text{for almost every }\xi.
\end{equation}
\end{theorem}

The answer is thus that the two parameters $\pG$ and $\delta$ are decoupled: any
positive density of good lines suffices, however small the enhancement
$\pG-\pc$. Nevertheless, some enhancement is necessary, since for $\delta=1$ the
model is critical oriented site percolation, which does not
percolate, see~\cite{BezuidenhoutGrimmett1990}.

\begin{remark}
\label{rem:quenched}
As mentioned
in~\cite[Sec.~2]{KestenSidoraviciusVares2022}, the event $\{\xi : \PP(C_{0}\text{ is infinite}\mid\xi)>0\}$ is a tail event of
the i.i.d.\ sequence $\xi$, so that~\eqref{eq:quenched} follows from
$\Theta(\pG,\pc,\delta)>0$ by Kolmogorov's zero--one law. Therefore, it suffices to prove the first assertion of
Theorem~\ref{thm:main}.
\end{remark}

Throughout the paper, $\pB=\pc$ and $\pG>\pc$ are fixed, and it is convenient to realize all the models simultaneously. Let
$(U_{z})_{z\in\Ze}$ be i.i.d.\ uniform variables on $[0,1]$, independent of
$\xi$. Given an environment $\xi$, set
\begin{equation}
  \label{eq:coupling}
  \eta^{\xi}_{z}:=\ind\{U_{z}<p_{\xi_{z_2}}\},
  \qquad
  \eta^{\mathrm{crit}}_{z}:=\ind\{U_{z}<\pc\},
\end{equation}
where $z_2$ is the height of $z$ and $p_{1}=\pB$, $p_{0}=\pG$. Since the
$U_{z}$ are independent of $\xi$, the conditional law of $\eta^{\xi}$ given
$\xi$ is $\Pxi$, and we realize the model by taking $\eta=\eta^{\xi}$.
Moreover,
\begin{equation}
  \label{eq:sandwich}
  \eta^{\mathrm{crit}}\ \le\ \eta^{\xi}
  \qquad\text{for every environment }\xi.
\end{equation}

\subsection{Sketch of the proof}

As noted by de Lima, Ungaretti and
Vares~\cite{LimaUngarettiVares2024}, standard enhancement techniques are not directly available in the oriented setting. Instead, the scheme of our proof is largely inspired by the work of Archer et al.\
\cite[Sec.~5]{ArcherEtAl2026}, with several
adaptations to our setting. They consider oriented site percolation enhanced by edges of length two, which,
on a given line, are either all open or all closed, independently of the other lines, and show that for any positive probability of enhancement the critical
model percolates.

In
Sections~\ref{sec:speeds} and~\ref{sec:strict}, we show that the right edge of
the process travels at a positive speed $\alpha(\delta)$ for every $\delta<1$; each good line
pushes the edge while the critical bad lines do not slow it down. After
fixing geometric conventions in Section~\ref{sec:geometry}, we deduce in
Section~\ref{sec:boxes} that boxes tilted along the directions
$(\pm\alpha(\delta),1)$ are crossed with probability close to $1$, and in
Section~\ref{sec:badtimes} that, by the box-crossing property
of Duminil-Copin,
Tassion and Teixeira~\cite{DuminilCopinTassionTeixeira2018}, rectangular boxes are crossed
at a constant cost uniformly in the environment, hence even during bad times. Finally, in Section~\ref{sec:renorm}, we combine these estimates into a
renormalization scheme which dominates the oriented
percolation with geometric defects of Hilário, Sá,
Sanchis and Teixeira~\cite{HilarioSaSanchisTeixeira2026} and
therefore percolates.

 \section{Edge speeds}
\label{sec:speeds}

We first introduce some notation. For $n\in\Z_{+}$, we identify the line $H_{n}$ with the set of $x$-coordinates
of its sites, $\{x\in\Z : x+n \text{ is even}\}$, and we write
\[
  H_{n}^{-}:=\{x\in H_{n} : x\le 0\},
  \qquad
  H_{n}^{+}:=\{x\in H_{n} : x\ge 0\},
  \qquad
  H_{n}^{b}:=\{x\in H_{n} : x\le b\}.
\]
For $z,z'\in\Ze$, we write $z\to z'$ if there is an open oriented path from $z$
to $z'$. For $m\le n$ and $A\subseteq H_{m}$, we write $(A,m)\to(x,n)$ if
$(a,m)\to(x,n)$ for some $a\in A$, and we define
\[
  \reach_{m,n}(A) := \bigl\{x\in H_{n} : (A,m)\to(x,n)\bigr\},
  \qquad
  \reach_{n}(A):=\reach_{0,n}(A).
\]
Since the initial site of an open path is not required to be open, we have
\begin{equation}
  \label{eq:semigroup}
  \reach_{m,n}(A)=\reach_{k,n}\bigl(\reach_{m,k}(A)\bigr),
  \qquad m\le k\le n.
\end{equation}
The right and left edges of the process started from a half-line are
\[
  r_{n} := \max\reach_{n}(H_{0}^{-}),
  \qquad
  \ell_{n} := \min\reach_{n}(H_{0}^{+}),
  \qquad n\ge 0,
\]
with the conventions $\max\emptyset=-\infty$ and $\min\emptyset=+\infty$. Both
$H_{0}^{-}=\{0,-2,-4,\dots\}$ and $H_{0}^{+}=\{0,2,4,\dots\}$ contain the
origin, so that $r_{0}=\ell_{0}=0$.

We denote by $\mathcal{F}_{m}:=\sigma\bigl((\xi_{i})_{0\le i\le m},\,
(U_{z})_{z\in\Ze,\,z_{2}\le m}\bigr)$ the $\sigma$-field generated by the environment and
the uniform variables up to height $m$. Finally, since an open path started in
$H_{0}^{-}$ reaches height $n$ at an $x$-coordinate at most $n$, we have
$r_{n}\le n$, so that $a_{n}:=\EE[r_{n}]\in[-\infty,n]$ is well defined.

\begin{proposition}[Existence of the edge speeds]
\label{prop:speeds}
For every $\delta\in[0,1]$, there exists $\alpha(\delta)\in[-\infty,1]$ such that
\[
  \frac{r_{n}}{n}\xrightarrow[n\to\infty]{}\alpha(\delta)
  \qquad\text{and}\qquad
  \frac{\ell_{n}}{n}\xrightarrow[n\to\infty]{}-\alpha(\delta)
\]
almost surely under $\mathbb{P}$.
\end{proposition}

\begin{proof}
The reflection $(x,n)\mapsto(-x,n)$ preserves $\Ze$ together with its oriented
graph structure and, because the opening probability $p_{\xi_{n}}$ depends only
on the height~$n$, it preserves the law $\Pxi$ for every environment
$\xi$. Since it exchanges $H_{0}^{-}$ and $H_{0}^{+}$, the processes $(\ell_{n})_{n\ge0}$ and $(-r_{n})_{n\ge0}$ have the same law, and it suffices to treat the right edge.

By Lemma~\ref{lem:basic} below, whose proof does not rely on this section, the
set $\reach_{m}(H_{0}^{-})$ is almost surely non-empty, so that $r_{m}>-\infty$
almost surely. For $0\le m\le n$ let $X_{m,n}:=\max\reach_{m,n}\bigl(H_{m}^{r_{m}}\bigr)-r_{m}$.

By definition,
$\reach_{m}(H_{0}^{-})\subseteq H_{m}^{r_{m}}$, therefore
\begin{align}
  r_{n}
  &= \max\reach_{n}(H_{0}^{-}) \notag\\
  &= \max\reach_{m,n}\bigl(\reach_{m}(H_{0}^{-})\bigr)
     && \text{by }\eqref{eq:semigroup} \notag\\
  &\le \max\reach_{m,n}\bigl(H_{m}^{r_{m}}\bigr)
     && \text{since }\reach_{m,n}(\cdot)\text{ is non-decreasing} \notag\\
  &= r_{m}+X_{m,n}.
  \label{eq:X}
\end{align}
For $m=0$, $r_{0}=0$ and $H_0^{r_{0}}=H_{0}^{-}$, so
$X_{0,n}=r_{n}$.

For $b\in H_{m}$ consider the space--time translation $\theta_{b,m}:(x,i)\longmapsto(x+b,\,i+m)$. Since $b\in H_{m}$, $(x+b)+(i+m)$ is even exactly
when $x+i$ is. Therefore the map $\theta_{b,m}$ preserves $\Ze$, and it sends
$H_{0}^{-}$ onto $H_{m}^{b}$. The environment and the uniform variables at heights $>m$ are independent of
$\mathcal{F}_{m}$, and their joint law is invariant under $\theta_{b,m}$, since $\xi$ is
stationary in the time variable and the variables $(U_{z})$ are i.i.d. Moreover, since the first site of an open path is not required to be
open, $\reach_{m,n}$ is a function of the configuration above height $m$. Applying this with the $\mathcal{F}_{m}$-measurable shift $b=r_{m}\in H_{m}$, we obtain that, conditionally on $\mathcal{F}_{m}$,
\[
  (X_{m,m+j})_{j\ge1}\ \overset{\text{law}}{=}\ (r_{j})_{j\ge1}.
\]

For the almost sure convergence, we apply Liggett's subadditive ergodic theorem
\cite[Theorem~1.10]{Liggett1985} to the array $(X_{m,n})$. Hypothesis (1.7) is
exactly~\eqref{eq:X}, since $X_{0,m}=r_{m}$. Hypothesis (1.8) is the above
identification of the conditional law of $(X_{m,m+j})_{j\ge1}$, which does not
depend on $m$. As that conditional law does not depend on the conditioning
either, $X_{m,n}$ is independent of $\mathcal{F}_{m}$. The variables
$(X_{jk,(j+1)k})_{j\ge1}$ are therefore i.i.d., which is (1.9). As Liggett
remarks~\cite[p.~1280]{Liggett1985}, (1.3) can be weakened to
$\EE[X_{0,1}^{+}]<\infty$, at the price of allowing the limit in (1.4) to be
$-\infty$ and of dropping the $L^{1}$ convergence in (1.5). This weaker
assumption holds here since $X_{0,1}=r_{1}\le1$. The processes in (1.9) being
i.i.d., hence ergodic, (1.4) and (1.5) give that $r_{n}/n$ converges almost
surely to the constant
\[
  \alpha(\delta):=\lim_{n\to\infty}\frac{a_{n}}{n}=\inf_{n\ge1}\frac{a_{n}}{n}
  \ \in\ [-\infty,1].
\]
\end{proof}

\section{The strict inequality}
\label{sec:strict}

The main result of this section is the following.
\begin{proposition}[Positivity of the edge speeds]
\label{prop:strict}
For every $\pG>\pc$ and every $\delta\in[0,1]$,
\begin{equation}
  \label{eq:strict}
  \alpha(\delta)\ \ge\ 2\,(\pG-\pc)\,(1-\delta).
\end{equation}
In particular, for every $\delta\in[0,1)$, $\alpha(\delta)>0$.
\end{proposition}

We first state three lemmas used to prove Proposition~\ref{prop:strict}.
\begin{lemma}
\label{lem:basic}
Let $\xi$ be an environment, let $m\le n$, and let $A\subseteq H_{m}$ be
infinite and not bounded from below. Then $\Pxi(\reach_{m,n}(A)=\emptyset)=0$
and, for every integer $k\ge0$ and every $b\in H_{m}$,
\[
  \Pxi\bigl(\max\reach_{m,n}(H_{m}^{b})\le b-2k\bigr)
  \ \le\ \bigl(1-\pc^{\,n-m}\bigr)^{k}.
\]
In particular, $\Exi|\max\reach_{m,n}(H_{m}^{b})|<\infty$ and $\Exi|r_{n}|<\infty$.
\end{lemma}

\begin{proof}
For $x\in H_{m}$, let $D(x):=\{(x+i,m+i):1\le i\le n-m\}$ be the ascending
diagonal issued from $(x,m)$. Diagonals issued from distinct points are disjoint and each is
entirely open with probability at least $\pc^{\,n-m}$ because
$p_{\xi_{i}}\ge\pc$. If $x\in A$ and $D(x)$ is open then
$x+(n-m)\in\reach_{m,n}(A)$.

Since $A$ is infinite and not bounded from below, we may pick
$x_{0}>x_{1}>\dots$ in $A$. The diagonals $D(x_{0}),D(x_{1}),\dots$ being
disjoint, the corresponding events are conditionally independent given $\xi$, so
that the probability that none of $D(x_{0}),\dots,D(x_{j-1})$ is open is at most
$(1-\pc^{\,n-m})^{j}$. Letting $j\to\infty$ shows that
$\reach_{m,n}(A)\ne\emptyset$ almost surely.

The same argument applied to the $k$ points $b,b-2,\dots,b-2(k-1)$ of
$H_{m}^{b}$ gives that if $D(b-2i)$ is open for some
$0\le i\le k-1$, then $b-2i+(n-m)\in\reach_{m,n}(H_{m}^{b})$, thus
\[
  \max\reach_{m,n}(H_{m}^{b})\ \ge\ b-2(k-1)+(n-m)\ >\ b-2k ,
\]
so that the event $\{\max\reach_{m,n}(H_{m}^{b})\le b-2k\}$ forces all $k$
diagonals to be closed.

Finally, $\max\reach_{m,n}(H_{m}^{b})\le b+(n-m)$, since the $x$-coordinate increases
by at most $1$ at each step, and the above bound gives an exponential tail on
the left, thus $\Exi|\max\reach_{m,n}(H_{m}^{b})|<\infty$, uniformly in $\xi$.
\end{proof}

\begin{lemma}
\label{lem:gain}
Let $\xi$ be an environment, let $m\le n$, and let $A\subseteq B\subseteq H_{m}$
with $A$ infinite and not bounded from below and such that
$x^{*}:=\max B<\infty$ satisfies $x^{*}>\max A$. Then
\[
  \Exi\bigl[\max\reach_{m,n}(B)-\max\reach_{m,n}(A)\bigr]\ \ge\ 2 .
\]
\end{lemma}

\begin{proof}
By Lemma~\ref{lem:basic} applied to $A$, the set $\reach_{m,n}(A)$ is almost
surely non-empty, and the same holds for $\reach_{m,n}(B)\supseteq\reach_{m,n}(A)$. Therefore, both maxima are almost surely finite, being bounded above by
$x^{*}+(n-m)$, and the differences below are well defined.
Since $\max A<x^{*}$ we have $A\subseteq B\setminus\{x^{*}\}\subseteq
H_{m}^{x^{*}-2}$, and therefore
\begin{align*}
  \max\reach_{m,n}(B)-\max\reach_{m,n}(A)
  &\ \ge\ \max\reach_{m,n}(B)-\max\reach_{m,n}\bigl(B\setminus\{x^{*}\}\bigr)\\
  &\ =\ \Bigl(\max\reach_{m,n}(\{x^{*}\})
        -\max\reach_{m,n}\bigl(B\setminus\{x^{*}\}\bigr)\Bigr)^{+}\\
  &\ \ge\ \Bigl(\max\reach_{m,n}(\{x^{*}\})-\max\reach_{m,n}(H_{m}^{x^{*}-2})\Bigr)^{+}\\
  &\ =\ \max\reach_{m,n}(H_{m}^{x^{*}})-\max\reach_{m,n}(H_{m}^{x^{*}-2}),
\end{align*}
using $\reach_{m,n}(B)=\reach_{m,n}(\{x^{*}\})\cup\reach_{m,n}(B\setminus\{x^{*}\})$
for the first equality and $H_{m}^{x^{*}}=H_{m}^{x^{*}-2}\cup\{x^{*}\}$ for the last one.
Since $H_{m}^{x^{*}-2}=H_{m}^{x^{*}}-2$ and $\Pxi$ is invariant under the horizontal
translation $x\mapsto x-2$, we have
\[
  \max\reach_{m,n}(H_{m}^{x^{*}-2})
  \ \overset{\text{law}}{=}\ \max\reach_{m,n}(H_{m}^{x^{*}})-2 ,
\]
and both terms being integrable by Lemma~\ref{lem:basic}, we can take expectations
and obtain the desired result.
\end{proof}

We now want to turn this result into a gain per good line. Fix $n\ge1$. For
$1\le m\le n$, let $\xi^{m,G}$ (resp. $\xi^{m,B}$) be the environment obtained
from $\xi$ by forcing the line $H_{m}$ to be good (resp. bad), the other lines
being unchanged, and let $\Delta_{m}r_{n}:=r_{n}\bigl(\xi^{m,G}\bigr)-r_{n}\bigl(\xi^{m,B}\bigr)$,
both using the same variables $(U_{z})$. By monotonicity of
the coupling~\eqref{eq:coupling}, $\Delta_{m}r_{n}\ge0$.

\begin{lemma}
\label{lem:line}
For every $1\le m\le n$ and every $\delta\in[0,1]$,
\[
\EE[\Delta_{m}r_{n}]\ge2(\pG-\pc).
\]
\end{lemma}

\begin{proof}
Let $\widehat{\mathcal{F}}_{m}:=\sigma\bigl((\xi_{i})_{i\ne m},\,(U_{z})_{z\in\Ze,\,z_{2}\le m}\bigr),$ and let $G_{m}$ and $B_{m}$ be the sets $\reach_{m}(H_{0}^{-})$
computed under $\xi^{m,G}$ and $\xi^{m,B}$ respectively. The two environments
differ only on $H_{m}$, so the occupation variables are the same at every height
but $m$. In particular, $r_{m-1}$ is identical for both.

Let $v:=(r_{m-1}+1,m)$. On $\{\pc\le U_{v}<\pG\}$ the site $v$ is open under
$\xi^{m,G}$ and closed under $\xi^{m,B}$. Thus, on this event $\max G_{m}\ge r_{m-1}+1$ and $\max B_{m}\le r_{m-1}-1$, so $\max G_{m}> \max B_{m}$.
Conditionally on $\mathcal{F}_{m-1}$ the site $v$ is deterministic, and $U_{v}$
is at height $m$, so it is independent of $\mathcal{F}_{m-1}$. Thus,
$\PP(\pc\le U_{v}<\pG\mid\mathcal{F}_{m-1})=\pG-\pc$ almost surely, and
\begin{equation}
  \label{eq:pivot}
  \PP\bigl(\max G_{m}>\max B_{m}\bigr)\ \ge\ \pG-\pc .
\end{equation}

For every $k$, applying Lemma~\ref{lem:basic} to $\{x\in H_{0}:x\le-2k\}$ shows
that $B_{m}$ almost surely contains a point at $x$-coordinate at most $-2k+m$,
since the $x$-coordinate increases by at most $1$ per step. Intersecting over
$k$, the sets $B_{m}\subseteq G_{m}$ are almost surely infinite and not bounded
from below, while $\max G_{m}\le r_{m-1}+1<\infty$. Since the first site of an
open path is not required to be open, $\reach_{m,n}$ depends only on heights
$m+1,\dots,n$, where the two occupation configurations are the same. Therefore, it is the
same map in both configurations and, by~\eqref{eq:semigroup}, $\Delta_{m}r_{n}=\max\reach_{m,n}(G_{m})-\max\reach_{m,n}(B_{m})$.

Conditionally on $\widehat{\mathcal{F}}_{m}$, the sets $B_{m}\subseteq G_{m}$
are fixed, the environment above height $m$ is fixed, and $\reach_{m,n}$ is
generated by the variables $U_{z}$ at heights $>m$, so that its conditional law
is $\Pxi$. Applying Lemma~\ref{lem:gain} conditionally on
$\widehat{\mathcal{F}}_{m}$, and using $\Delta_{m}r_{n}\ge0$ on the
complementary event,
\[
  \EE\bigl[\Delta_{m}r_{n}\mid\widehat{\mathcal{F}}_{m}\bigr]
  \ \ge\ 2\cdot\ind\{\max G_{m}>\max B_{m}\}.
\]
Taking expectations and using~\eqref{eq:pivot} gives the claim.
\end{proof}

\begin{proof}[Proof of Proposition~\ref{prop:strict}]
Write $q:=1-\delta$ for the density of good lines and see $a_{n}$ as a
function of $q$. The edge $r_{n}$ depends on the environment only through
$(\xi_{1},\dots,\xi_{n})$, so for $w\in\{0,1\}^{n}$ we write
$\Exi[r_{n}]=\EE^{w}[r_{n}]$ whenever $\xi$ agrees with $w$ on the first $n$
lines. Let $\mu_{q}$ be the law of a single line, $\mu_{q}(0)=q$ and
$\mu_{q}(1)=1-q$, so that
\[
  a_{n}=\sum_{w\in\{0,1\}^{n}}\ \prod_{i=1}^{n}\mu_{q}(w_{i})\ \EE^{w}[r_{n}].
\]
This is a polynomial in $q$. Indeed, the sum is finite, and the coefficients
$\EE^{w}[r_{n}]$ are finite by Lemma~\ref{lem:basic} and do not depend on $q$.
Since $\mathrm{d}\mu_{q}(0)/\mathrm{d}q=1$ and
$\mathrm{d}\mu_{q}(1)/\mathrm{d}q=-1$, the product rule gives
\[
  \frac{\mathrm{d}a_{n}}{\mathrm{d}q}
  =\sum_{m=1}^{n}\ \sum_{w\in\{0,1\}^{n}}
    \varepsilon(w_{m})\prod_{i\ne m}\mu_{q}(w_{i})\ \EE^{w}[r_{n}],
  \qquad \varepsilon(0)=1,\ \varepsilon(1)=-1 .
\]
For fixed $m$, the weight $\prod_{i\ne m}\mu_{q}(w_{i})$ does not depend on
$w_{m}$. It only involves the status of the lines other than $H_{m}$. Pairing
each environment $w$ with the one obtained by changing the status of $H_{m}$,
and using $\varepsilon(0)=1$, $\varepsilon(1)=-1$, the inner sum becomes
\[
  \sum_{v}\ \prod_{i\ne m}\mu_{q}(v_{i})
  \Bigl(\EE^{\xi^{m,G}}[r_{n}]-\EE^{\xi^{m,B}}[r_{n}]\Bigr),
\]
where $v$ runs over the $2^{n-1}$ environments of the lines
$H_{i}$, $i\le n$, $i\ne m$, and $\xi$ is any environment agreeing with $v$ off
$m$. Since $\xi^{m,G}$ and $\xi^{m,B}$ do not depend on $\xi_{m}$, this is
exactly $\EE[\Delta_{m}r_{n}]$. Notice that we recover Russo's formula. Moreover, Lemma~\ref{lem:line} gives
\[
  \frac{\mathrm{d}a_{n}}{\mathrm{d}q}
  =\sum_{m=1}^{n}\EE\bigl[\Delta_{m}r_{n}\bigr]
  \ \ge\ 2\,(\pG-\pc)\,n
  \qquad\text{for every } q\in[0,1].
\]
It remains to determine the constant of integration. At $q=0$ every line is bad,
so that $a_{n}=\EE_{\pc}[r_{n}]$, the homogeneous critical model, whose edge
speed vanishes, $\lim_{n\to\infty}\EE_{\pc}[r_{n}]/n=0$,
by the classical result of Durrett and Griffeath~\cite{DurrettGriffeath1983}. By
Proposition~\ref{prop:speeds} this limit is also an infimum, so that
$\EE_{\pc}[r_{n}]\ge0$ for every $n\ge1$. Integrating therefore gives
$a_{n}\ge2(\pG-\pc)\,q\,n$ for every $n\ge1$, and taking the infimum over $n$
gives~\eqref{eq:strict}.
\end{proof}

\begin{remark}
The factor 2 in the lower bound of $\alpha(\delta)$ in Proposition~\ref{prop:strict} may be surprising since $\alpha(\delta) \leq 1$, but it is in fact not a contradiction. In the classical homogeneous oriented percolation model on $\Ze$, there are exactly $2^n$ oriented paths of length $n$ issued from a given site and the $n$ sites of a given path are distinct and independently open with probability $p$. Therefore, $\PP_p(\exists$ an open oriented path of length $n$ from the origin$)\le (2p)^n$, so that $\pc\ge 1/2$. Hence, $2(\pG-\pc)\le 1$, and~\eqref{eq:strict} is consistent with
$\alpha(\delta)\le 1$.
\end{remark}

\begin{remark}
\label{rem:why-critical}
In Sections~\ref{sec:speeds} and~\ref{sec:strict}, the criticality of the bad lines is used only through the constant of
integration. Indeed, Lemmas~\ref{lem:basic},~\ref{lem:gain} and~\ref{lem:line}
use the standing assumption $\pB=\pc$, $\pG>\pc$ only through $p_{\xi_{i}}\ge p_{B}>0$: replacing $\pc$
by an arbitrary $p_{B}\in(0,\pG)$ in Lemma~\ref{lem:basic}, and the
event of Lemma~\ref{lem:line} by $\{p_{B}\le U_{v}<\pG\}$, the same argument
gives
\[
  \frac{\mathrm{d}a_{n}}{\mathrm{d}q}\ \ge\ 2\,(\pG-p_{B})\,n
  \qquad\text{and hence}\qquad
  \alpha(\delta)\ \ge\ \alpha_{\mathrm{hom}}(p_{B})+2\,(\pG-p_{B})(1-\delta),
\]
where $\alpha_{\mathrm{hom}}(p)$ denotes the right edge speed of the
homogeneous model with parameter $p$. Proposition~\ref{prop:strict} is the case
$p_{B}=\pc$, where $\alpha_{\mathrm{hom}}(\pc)=0$ by~\cite{DurrettGriffeath1983}.

For $p_{B}<\pc$, $\alpha_{\mathrm{hom}}(p_{B})=-\infty$ (see~\cite[Sec.~7]{Durrett1984}) and the estimate is trivial for every
$\delta<1$. This is why percolation for $p_{B}<\pc$ and $\delta$ small
in~\cite{KestenSidoraviciusVares2022} requires different arguments. In this sense,
$\pc$ is the smallest value of $p_{B}$ for which the present arguments say
something, and the only one for which the conclusion is not already trivial,
since for $p_{B}>\pc$ percolation holds even at $\delta=1$.
\end{remark}

\section{Geometric conventions}
\label{sec:geometry}
Before resuming the proof, we introduce the geometric objects used below.
For $A,B,C\subseteq\Ze$, let $A\to[B]\,C$ denote the event that some site of $A$
is connected to some site of $C$ by an open oriented path contained in $B$. Such an oriented path is called a \emph{crossing path}.
Given $u,v\in\R^{2}$ with $\det(u,v)>0$, write
\[
  R(u,v):=\bigl([0,1)u+[0,1)v\bigr)\cap\Ze
\]
for the parallelogram generated by $u$ and $v$. All the parallelograms below have a horizontal generator $u=(a,0)$ with $a\ge2$ giving the width of the parallelogram
and a generator $v=(b,c)$ with $c \in \N^*$ giving the length, which is the number of lines the parallelogram meets, and
$b\in\R$ giving the tilt of the parallelogram. Translations of $R(u,v)$ by
vectors of $\R^{2}$ are used below with the convention that the intersection
with $\Ze$ is taken after the translation: for $z\in\R^{2}$, $z+R(u,v)$ stands
for $\bigl(z+[0,1)u+[0,1)v\bigr)\cap\Ze$. We call \emph{box} any such translate $z+R(u,v)$, $z$ giving its position, and
\emph{rectangular box} a box with no tilt, $b=0$.

For such a box $R$, we write $R_{h}:=\{x\in\Z:(x,h)\in R\}$ for every height $h$ of $R$, and $\partial^{\downarrow}R$ (resp. $\partial^{\uparrow}R$)
for the set of sites of its lowest (resp. highest) line. We also write
\[
  \partial^{\leftarrow}R:=\bigl\{(\min R_{h},h)\bigr\}_{h},
  \qquad
  \partial^{\rightarrow}R:=\bigl\{(\max R_{h},h)\bigr\}_{h},
\]
for the sets of sites of $R$ having, at their own height, no site of $R$ to
their left (resp. right). It is the left (resp. right) side of $R$.

The three crossing events we use are
\[
  C^{\uparrow}(R):=\bigl\{\partial^{\downarrow}R\to[R]\ \partial^{\uparrow}R\bigr\},
\]
\[
  C^{\rightarrow}(R):=\bigl\{\partial^{\leftarrow}R\to[R]\ \partial^{\rightarrow}R\bigr\},
  \qquad
  C^{\leftarrow}(R):=\bigl\{\partial^{\rightarrow}R\to[R]\ \partial^{\leftarrow}R\bigr\},
\]
namely the events that $R$ is crossed from its base to its top, from its left side to its right
side, and \emph{vice versa}. A crossing path is thus contained in $R$, both of
its endpoints included, but it crosses $R$ in the sense that it could not go one step further in the considered direction without leaving $R$.

Real translations of the boxes do not preserve $\Ze$ in general. The box
$z+R(u,v)$ may contain one site more or one site less per line than $R(u,v)$, so that its crossing probability may differ from that of $R(u,v)$. As in
\cite[Sec.~5.3]{ArcherEtAl2026}, every crossing estimate below is stated for one
box $R(u,v)$ and is to be understood uniformly for all its translates
$z+R(u,v)$ with $z\in\R\times[0,\infty)$. The proofs are written for one
position, the extension to the others requiring only minor adaptations, which
we do not write. Moreover, each crossing estimate is stated for a box at its lowest position, so that we only apply the convention above to translations by vectors $z\in\R\times[0,\infty)$, which shift the box upwards and thus keep it in $\Ze$.

\section{Crossing boxes in the supercritical regime}
\label{sec:boxes}

Proposition~\ref{prop:strict} states that the right edge issued from $H_{0}^{-}$
travels at a positive speed $\alpha(\delta)$, and, by symmetry, that the left edge
issued from $H_{0}^{+}$ travels at speed $-\alpha(\delta)$. The cone of
propagation, that is, the intersection of the two reachable regions, therefore has
non-empty interior, of opening $2 \,\alpha(\delta)$ around the vertical. The purpose of this section is to show that a box whose base has arbitrarily small relative width, and whose other two sides follow one of the two boundary directions $(\alpha(\delta),1)$ or $(-\alpha(\delta),1)$ of the cone, is crossed from bottom to top with probability tending to $1$ as its size grows.

From now on, $\delta\in[0,1)$ is fixed and we write $\alpha:=\alpha(\delta)$,
which is positive by Proposition~\ref{prop:strict}.

\begin{proposition}[Supercritical box-crossing]
\label{prop:boxsuper}
For every $\rho>0$ and
every $\varepsilon>0$, there exists $n_{0}$ such that, for every $n\ge n_{0}$,
letting
\[
  u=(\rho n,0),
  \qquad
  v^{+}=(\alpha n,n),
  \qquad
  v^{-}=(-\alpha n,n),
\]
we have
\[
  \PP\bigl(C^{\uparrow}(R(u,v^{+}))\bigr)>1-\varepsilon
  \qquad\text{and}\qquad
  \PP\bigl(C^{\uparrow}(R(u,v^{-}))\bigr)>1-\varepsilon.
\]
\end{proposition}

Our proof of Proposition~\ref{prop:boxsuper} requires the following lemma.

\begin{lemma}
\label{lem:tail}
Let $\rho>0$. There
exist $c>0$ and $n_{0}\in\N$ such that, for every $n\ge n_{0}$, every
$k\le m\le n$ and every $b\in H_{k}$,
\begin{equation}
  \label{eq:tailk}
  \PP\bigl(\max\reach_{k,m}(H_{k}^{b})>b+\alpha(m-k)+\rho n\bigr)<e^{-cn}.
\end{equation}
\end{lemma}

\begin{proof}
By the translation $\theta_{b,k}$ of the proof of Proposition~\ref{prop:speeds},
\[
  \bigl(\max\reach_{k,k+j}(H_{k}^{b})-b\bigr)_{j\ge0}
  \ \overset{\text{law}}{=}\ (r_{j})_{j\ge0},
\]
so the probability in~\eqref{eq:tailk} equals $\PP(r_{j}>\alpha j+\rho n)$ with
$j:=m-k\le n$. It therefore suffices to bound $\PP(r_{j}>\alpha j+\rho n)$ uniformly
over $j\le n$.

Since $a_{M}/M\to\alpha$, we can
fix $M\in\N$ with
\begin{equation}
  \label{eq:choiceM}
  a_{M}<\Bigl(\alpha+\frac{\rho}{4}\Bigr)M .
\end{equation}
Write $j=Ml+t$ with $0\le t<M$ and $l\le n/M$. By~\eqref{eq:X},
$r_{j}\le r_{Ml}+X_{Ml,j}$, and $X_{Ml,j}\le j-Ml=t$ since the $x$-coordinate
increases by at most $1$ at each level. Here, our model differs
from~\cite[Lemma~18]{ArcherEtAl2026}, where the increments of the right edge are
only dominated by geometric variables. Hence, if $r_{j}>\alpha j+\rho n$, then
\[
  r_{Ml}>\alpha j+\rho n-t\ \ge\ \alpha Ml+\rho n-M\ \ge\ \alpha Ml+\frac{\rho}{2}n
\]
as soon as $n\ge2M/\rho$.

Iterating~\eqref{eq:X} along the intervals $[M(i-1),Mi]$ and using $r_{0}=0$
gives
\[
  r_{Ml}\ \le\ \sum_{i=1}^{l}X_{M(i-1),Mi},
\]
where, by the proof of Proposition~\ref{prop:speeds}, the variables
$X_{M(i-1),Mi}$, $i\ge1$, are i.i.d.\ with the law of $r_{M}$. Define
\[
  Z_{i}:=X_{M(i-1),Mi}-\Bigl(\alpha+\frac{\rho}{4}\Bigr)M ,
  \qquad
  S_{l}:=\sum_{i=1}^{l}Z_{i}\ \ge\ r_{Ml}-\Bigl(\alpha+\frac{\rho}{4}\Bigr)Ml.
\]
On the event $\{r_{j}>\alpha j+\rho n\}$ we thus have
$S_{l}>\rho n/2-\rho Ml/4\ge\rho n/4$, since $Ml\le n$.
If $l=0$, then $S_{0}=0$ and this event is empty.

Since $r_{M}\le M$, the random variables $Z_{i}$ are bounded from above, and they
are integrable, by Lemma~\ref{lem:basic}, with $\EE[Z_{1}]<0$
by~\eqref{eq:choiceM}. So there exists $\lambda>0$ such that
$\EE[e^{\lambda Z_{1}}]<1$. For $l\ge1$, Chernoff's bound then gives
\[
  \PP\Bigl(S_{l}>\frac{\rho}{4}n\Bigr)
  \le e^{-\lambda\rho n/4}\,\EE\bigl[e^{\lambda Z_{1}}\bigr]^{l}
  <e^{-\lambda\rho n/4}.
\]
This proves~\eqref{eq:tailk} with $c:=\lambda\rho/4$ and $n_{0}\ge2M/\rho$.
\end{proof}

\begin{proof}[Proof of Proposition~\ref{prop:boxsuper}]
The reflection $(x,n)\mapsto(-x,n)$ preserves $\Ze$ together with its oriented
graph structure, it preserves $\PP$ and maps $R(u,v^{+})$ onto a
translate of $R(u,v^{-})$. Therefore, the second inequality follows from the first one.

By the convention of Section~\ref{sec:geometry}, it suffices to consider $R:=(-\rho n/2,0)+R(u,v^{+})$. Note that $\rho n\ge2$ for $n$ large, so that the lowest and highest lines of $R$
contain at least one site and $C^{\uparrow}(R)$ cannot be empty.
Consider the event
\[
  A_{n}:=\Bigl\{\,\alpha m-\frac{\rho}{4}n\ \le\ r_{m}\ \le\ \alpha m+\frac{\rho}{4}n
    \quad\text{for every } m\le n-1\,\Bigr\}.
\]
By Proposition~\ref{prop:speeds}, $r_{m}/m\to\alpha$ almost surely, hence
$\PP(A_{n})\to1$ as $n\to\infty$.

By Lemma~\ref{lem:basic} applied to $H_{0}^{-}$, the set $\reach_{n-1}(H_{0}^{-})$
is almost surely non-empty. Thus, there is an open path
$\gamma=((x_{0},0),\dots,(x_{n-1},n-1))$ with $x_{0}\in H_{0}^{-}$ and $x_{n-1}=r_{n-1}$.
If multiple such paths exist, we choose one using some arbitrary procedure. On $A_{n}$, since $x_{i}\in\reach_{i}(H_{0}^{-})$ for every $i$, we have $x_{i}\le r_{i}\le\alpha i+\rho n/4$, so that $\gamma$ never exits $R$ through its right side. Its endpoint $x_{n-1}=r_{n-1}$ satisfies $\alpha(n-1)-\rho n/4\le x_{n-1}\le\alpha(n-1)+\rho n/4$, so that it belongs to the top side of $R$. Notice that, in contrast with
the proof of~\cite[Lemma~8]{ArcherEtAl2026}, we do not need to treat heights of different parities separately, since the bound holds at every height rather than at even ones only.

It remains to exclude the case where $\gamma$ exits $R$ through its left side (or starts to the left of the bottom side of $R$). Suppose that it does, and let $0 \leq i< n-1$ be the first height at which $x_{i}<\alpha i-\rho n/2$. Since $\gamma$ ends at
$x_{n-1}\ge\alpha(n-1)-\rho n/4$, the portion of $\gamma$ above height $i$ is an
open path from $(x_{i},i)$ to height $n-1$ reaching an $x$-coordinate at least
$\alpha(n-1)-\rho n/4$, so that
\[
  \max\reach_{i,n-1}\bigl(H_{i}^{\,b_{i}}\bigr)
  \ \ge\ \alpha(n-1)-\frac{\rho}{4}n,
  \qquad
  b_{i}:=\max\Bigl\{x\in H_{i}:x<\alpha i-\frac{\rho}{2}n\Bigr\}.
\]
Since $b_{i}<\alpha i-\rho n/2$, we deduce
\[
  \max\reach_{i,n-1}\bigl(H_{i}^{\,b_{i}}\bigr)
  \ >\ b_{i}+\alpha(n-1-i)+\frac{\rho}{4}n.
\]
Each $b_{i}$ being deterministic, Lemma~\ref{lem:tail}, applied with $m=n-1$
and $\rho/4$ in place of $\rho$, and a union bound over $0\le i<n-1$ bound by
$n\,e^{-cn}$ the probability that, on $A_{n}$, $\gamma$ exits $R$ through its
left side. Hence, on $A_{n}$ minus an event of probability at most
$n\,e^{-cn}$, the path $\gamma$ stays in $R$ and crosses it from the base to the
top, so that $\PP(C^{\uparrow}(R))\ge\PP(A_{n})-n\,e^{-cn}>1-\varepsilon$ for
$n$ large enough.
\end{proof}

\section{Crossing bad times}
\label{sec:badtimes}

Proposition~\ref{prop:boxsuper} says that, on average over the environment, boxes can be crossed by an inside path with arbitrarily high probability. Nevertheless, it says nothing about what happens inside a time interval of consecutive bad lines. Such bad times do occur. Indeed, for every $k$, the event that the lines $H_{n+1},\dots,H_{n+k}$ are
all bad has probability $\delta^{k}>0$, and occurs for infinitely many $n$ almost surely. During these bad times, the model is exactly homogeneous oriented site
percolation at the critical threshold $\pc$. Therefore, no estimate of the form $1-\varepsilon$ can hold. Nevertheless, we can rely on less, it is enough for the renormalization of
Section~\ref{sec:renorm} that the crossing probability is bounded below by a
constant, uniformly in the scale of the box. This is what the box-crossing property
of~\cite{DuminilCopinTassionTeixeira2018} for critical two-dimensional oriented
percolation provides at $\pc$.

The reflection $(x,n)\mapsto(-x,n)$ preserves $\Ze$ together with its oriented
graph structure, it preserves $\PP_{\pc}$ and exchanges $C^{\rightarrow}(R(u,v))$
and $C^{\leftarrow}(R(u,v))$ up to a translation. Therefore, any bound proved for one holds for the
other.

\begin{theorem}[{\cite[Theorem~1.3, Proposition~4.2, Remark~4.4]{DuminilCopinTassionTeixeira2018}}]
\label{thm:DuminilCopinTassionTeixeira2018}
There exists $\varepsilon>0$ such that, for any $n \in \N$ large enough, there exists
$w_{n} \in [\varepsilon n^{2/5},n^{1-\varepsilon}]\cap\mathbb{Z}$ such that
\[
  \PP_{\pc}\bigl(C^{\rightarrow}(R(3u,v))\bigr)\ \ge\ \varepsilon,
  \qquad
  \PP_{\pc}\bigl(C^{\uparrow}(R(u,3v))\bigr)\ \ge\ \varepsilon,
  \qquad
  \PP_{\pc}\bigl(C^{\leftarrow}(R(3u,v))\bigr)\ \ge\ \varepsilon,
\]
with $u=(w_{n},0)$ and $v=(0,n)$.
\end{theorem}

Decreasing $\varepsilon$ keeps the statement true, the interval allowed for
$w_{n}$ becoming larger and the lower bounds weaker, so we assume
$\varepsilon\le1/2$ from now on.

\begin{remark}
Notice that the boxes above are rectangular, while those of
Proposition~\ref{prop:boxsuper} are tilted along the directions
$(\pm\alpha(\delta),1)$. For the classical oriented site percolation model at $\pc$, the edge
speed is $0$, so the process has no drift and we should not tilt the boxes.

Since $w_{n}\le n^{1-\varepsilon}$, the width of such a rectangular box grows strictly slower than its length. Unlike in the non-oriented case, the natural box at
$\pc$ is not $n\times n$ but $w_{n}\times n$, $w_{n}$ being
comparable to the typical width of a critical cluster reaching height $n$~\cite[Proposition~4.1(iii)]{DuminilCopinTassionTeixeira2018}.
\end{remark}

By the coupling of Section~\ref{sec:intro}, the bounds above also hold in our model.

\begin{corollary}
\label{cor:DuminilCopinTassionTeixeira2018quenched}
With $\varepsilon$, $n$, $w_n$, $u$ and $v$ as in Theorem~\ref{thm:DuminilCopinTassionTeixeira2018}, the
three lower bounds hold with $\Pxi$ in place of $\PP_{\pc}$, for every
environment $\xi$.
\end{corollary}

\begin{proof}
The three crossing events are increasing in the occupation configuration
$(\eta_{z})$, and $\eta^{\mathrm{crit}}$, which has law $\PP_{\pc}$, satisfies
$\eta^{\mathrm{crit}}\le\eta^{\xi}$ by~\eqref{eq:sandwich}. Any crossing realized
by $\eta^{\mathrm{crit}}$ is therefore realized by $\eta^{\xi}$, uniformly
in $\xi$.
\end{proof}

\section{Renormalization}
\label{sec:renorm}

We now combine the estimates of Sections~\ref{sec:boxes} and~\ref{sec:badtimes}
into a renormalization scheme, following~\cite[Sec.~5.5]{ArcherEtAl2026}. The renormalized model is an oriented site--bond percolation on the symmetric oriented
lattice, with steps $(\pm1,1)$, whose vertices correspond to the successive
positions of a path traveling inside the cone of propagation. Its edges are
realized by boxes along the two boundary directions of that cone, open
with probability close to $1$ by Proposition~\ref{prop:boxsuper}. That estimate,
however, is not available everywhere: it holds on the heights where the
environment behaves typically, and the others come in geometrically distributed
runs, each of which has to be climbed at a constant cost by means of
Corollary~\ref{cor:DuminilCopinTassionTeixeira2018quenched}. These runs are carried by the vertices. Site--bond renormalized models have already proved useful in~\cite[Sec.~4]{LimaSidoraviciusVares2023}, in an
unoriented setting where, unlike here, the bad layers are carried by the bonds. Once the cost of the vertices is
charged to the edges, this is the situation covered by Theorem~\ref{thm:HilarioSaSanchisTeixeira2026} of Hilário,
Sá, Sanchis and Teixeira on oriented percolation with geometric defects.

The renormalized lattice is therefore not fixed on $\Ze$: the heights of its \emph{floors}, that is of the successive horizontal levels of
its vertices, depend on where the estimate happens to be available, which is a
matter of the environment and hence random. However, the geometry of the construction is not random. Once the heights of the floors are fixed, the transport boxes and the corridors, and hence
the events declaring vertices and edges open, are determined. We therefore separate the two. In
Section~\ref{sec:renormlattice} we fix an arbitrary choice of the heights of the floors and carry out the whole construction for that choice. In
Section~\ref{sec:renormZe} we show that for such an arbitrary choice, an
infinite open path on the renormalized lattice gives an infinite open path on
$\Ze$, which is Proposition~\ref{prop:reconstruction}. Only in
Section~\ref{sec:renormperc} do we show that we can exhibit one choice of the heights of the floors for which the renormalized model percolates, which is
Proposition~\ref{prop:renormperc}. Then Theorem~\ref{thm:main} follows.

\subsection{Preliminaries}

Since all paths of $\Ze$ move by $(\pm1,1)$, for an open oriented path
$\gamma$ we write $\gamma(h)$ for the $x$-coordinate of the unique site of
$\gamma$ at height~$h$, defined for $h$ in the height range of $\gamma$. The
whole reconstruction below relies on the following intuitive
property.

\begin{lemma}
\label{lem:planar}
Let $\gamma_1,\gamma_2,\pi$ be open oriented paths of $\Ze$ and $h \leq h'$ be two heights at which all three paths are defined. 
\begin{enumerate}
\item[(i)] If
  $\gamma_1(h)\le\gamma_2(h)$ and $\gamma_1(h')\ge\gamma_2(h')$, then
    \[
    s:=\min\bigl\{l\ge h:\gamma_1(l)\ge\gamma_2(l)\bigr\}
  \]
  satisfies $h\le s\le h'$ and $\gamma_1(s)=\gamma_2(s)$. Thus
  $\gamma_1$ and $\gamma_2$ share a site, and the path obtained by following one path
  up to height $s$ and the other one afterwards is open.
\item[(ii)] If $\gamma_1(h)<\pi(h)<\gamma_2(h)$ and $\gamma_1(h')=\gamma_2(h')$, then $\pi$ shares a site with $\gamma_1$ or with $\gamma_2$ at some height
  of $(h,h']$.
\end{enumerate}
\end{lemma}

\begin{proof}
At the same height, two sites are at an even distance, so $\gamma_1-\gamma_2$ takes
even values and since both $\gamma_1$ and $\gamma_2$ change their
$x$-coordinate by $\pm1$ at each step, $\gamma_1-\gamma_2$ changes by
$0$ or $\pm2$ between consecutive heights. Since $h'$ belongs to the set defining $s$, we have $h\le s\le h'$. If
$s=h$, then $\gamma_1(h)-\gamma_2(h)$ is both non-positive and non-negative,
hence zero. Otherwise $\gamma_1-\gamma_2$ is negative, hence at most $-2$, at
$s-1$ by definition of $s$, and non-negative at $s$. Since it changes
by at most $2$, it vanishes at $s$. The last assertion of (i) is trivial. For (ii), if $\pi$ shared a site with neither $\gamma_1$ nor $\gamma_2$ on
$(h,h']$, then $\gamma_1-\pi$ and $\pi-\gamma_2$ would never vanish and therefore would be strictly negative on
$(h,h']$. Thus, $\gamma_1-\pi + \pi-\gamma_2 = \gamma_1-\gamma_2$ would also be negative on $(h,h']$, contradicting
$\gamma_1(h')=\gamma_2(h')$.
\end{proof}

\subsection{The renormalization scheme}

The construction below aims at producing an instance of the following
result of oriented percolation with geometric defects
established in~\cite{HilarioSaSanchisTeixeira2026} by Hilário, Sá, Sanchis and Teixeira.

\begin{theorem}[{\cite[Theorem~8.2]{HilarioSaSanchisTeixeira2026}}]
\label{thm:HilarioSaSanchisTeixeira2026}
Let $p,\theta\in(0,1)$ and
$\mathbb{V}:=\{(i,n)\in\Z\times\Z_{+}:i+n\ \text{even}\}$, with the two oriented
edges $\bigl((i,n),(i\pm1,n+1)\bigr)$ issued from each vertex, the second
coordinate playing the role of time. Let $\zeta=(\zeta_{n})_{n\ge0}$ be a
sequence of i.i.d.\ random variables with $\PP(\zeta_{n}=k)=(1-\theta)\theta^{k}$
for $k\in\N$. Conditionally on $\zeta$, declare each of the two edges issued
from $(i,n)$ to be open independently with probability $p^{1+\zeta_{n}}$, and
denote by $\mathbb{P}^{\zeta}_{p}$ the law of this process. There exists $c>0$
such that if $\theta\le c$ and $p\ge1-c$, then for almost every $\zeta$, under
$\mathbb{P}^{\zeta}_{p}$, there is an infinite open path starting at $(0,0)$
with positive probability.
\end{theorem}

In~\cite[Theorem~8.2]{HilarioSaSanchisTeixeira2026} the lattice is
$\{(i,j)\in\Z^{2}:i+j\ \text{even}\}$, with time as first coordinate and defects
indexed by $\Z$. Exchanging the two coordinates gives the statement above on
$\Z^{2}$, and the restriction to $\Z\times\Z_{+}$ is harmless since an open path
starting at $(0,0)$ only uses edges issued from the times $n\ge0$, hence only
involves $(\zeta_{n})_{n\ge0}$.

We proceed in three steps. In Section~\ref{sec:renormlattice} we construct a
candidate site--bond model. In Section~\ref{sec:renormZe} we prove that it
transfers to $\Ze$: an infinite open path of the renormalized lattice gives an
infinite open path of $\Ze$. In Section~\ref{sec:renormperc} we prove that the
candidate dominates the model above and, therefore, percolates.

\subsubsection{The renormalized lattice}
\label{sec:renormlattice}

We now define the grid, the transport boxes, and the corridors, and declare which
vertices and edges are open. Nothing is proved in this section.

Let $(X_{k})_{k\ge0}$ be an arbitrary sequence of random
variables with values in $\{0,1\}$ and with infinitely many ones almost surely. It is a selection procedure for the slabs, which fixes the heights of the floors: the slab $k$ is
\emph{retained} if $X_{k}=1$, and \emph{discarded} otherwise. No assumption is
made on its law, nor on its relation to the environment. The geometric
construction below, and the passage from the renormalized lattice to $\Ze$ in
Section~\ref{sec:renormZe}, hold for every selection procedure. Only in
Section~\ref{sec:renormperc} do we choose a specific one, retaining only some
favorable slabs, defined there, for which the renormalized model percolates.
The heights, regions, boxes and ladders defined below are functions of
$(X_{k})_{k\ge0}$ alone.

We begin by formally defining all the required objects. Figures~\ref{fig:goodvertex} and~\ref{fig:badvertex}, at the end of this section, illustrate them, except the ladder $\Lambda_{i,n}$, which is pictured in Figure~\ref{fig:pathbadvertex} in Section~\ref{sec:renormZe}. 

\emph{The grid.} The propagation involves the drift $\alpha$, so any relation between a horizontal and a
vertical length involves a fraction of $\alpha$. To avoid small fractions of $\alpha$, we fix an
integer $L\ge1$, which will be taken large and is to be understood as the scale of the grid, and set $\beta:=\alpha L$. Cut $\Ze$ into horizontal \emph{slabs} of length $\ell:=12L$,
the slab $k$ being the set of sites of $\Ze$ with height in
$[k\ell,(k+1)\ell)$, that is $\ell$ consecutive lines. Since $\ell$ is even,
the vertical translation by $(0,k\ell)$ preserves $\Ze$.

Set $\tau_{-1}:=-1$ and
\begin{equation}
  \label{eq:tau}
  \tau_{n}:=\inf\{k>\tau_{n-1}:X_{k}=1\},
  \qquad
  \kappa_{n}:=\tau_{n}-\tau_{n-1}-1,
  \qquad n\ge0 ,
\end{equation}
so that $\tau_{n}$ is the $n$-th retained slab and $\kappa_{n}$ the length of the
run of discarded slabs lying just below it. The renormalized lattice is the lattice $\mathbb{V}=\{(i,n)\in\Z\times\Z_{+}:i+n\ \text{even}\}$
of Theorem~\ref{thm:HilarioSaSanchisTeixeira2026}, with the two oriented edges
$\bigl((i,n),(i\pm1,n+1)\bigr)$ issued from each vertex, and we still write
$\mathbb{V}$ for the resulting oriented graph; the run $\kappa_{n}$
plays the role of the defect $\zeta_{n}$ there, affecting all the vertices of
the floor $n$. An element
$(i,n)$ of $\mathbb{V}$ is only a label, not coordinates in $\Ze$; it represents a region of $\Ze$,
which we also call the vertex $(i,n)$, and the edges are represented by the
\emph{transport boxes} defined below. The positions of these regions and transport boxes are what
we call the \emph{scheme} associated with $(X_{k})_{k\ge0}$.

The index $i$ fixes the $x$-coordinates; the vertex $(i,n)$ is centered on the column
$ig$, where
\begin{equation}
  \label{eq:g}
  g:=\text{the largest even integer}\ \le 10\beta .
\end{equation}
Since $g$ is even, the horizontal translation by $(ig,0)$ preserves $\Ze$ and
$\Pxi$. The $x$-coordinates depend on $i$ alone, not on $(X_{k})_{k\ge0}$, because discarded
slabs are climbed vertically and therefore cost length but no horizontal
displacement.

The index $n$, on the other hand, is not a height of $\Ze$ but a counter: it numbers the
retained slabs, it is the floor of the vertex $(i,n)$. Two heights of $\Ze$ are attached to it,
\begin{equation}
  \label{eq:heights}
  h_{n}:=(\tau_{n-1}+1)\,\ell,
  \qquad
  \tilde h_{n}:=\tau_{n}\,\ell,
  \qquad\text{so that}\qquad
  \tilde h_{n}-h_{n}=\kappa_{n}\,\ell.
\end{equation}
The paths of the floor $n-1$ arrive at the height $h_{n}=\tilde h_{n-1}+\ell$,
those of the floor $n$ leave from the height $\tilde h_{n}$, and the
$\kappa_{n}$ discarded slabs in between have to be climbed. Therefore, we define the vertex $(i,n)$ as the region
\[
  V_{i,n}:=(ig-5\beta,\,h_{n})+R\bigl((10\beta,0),(0,\kappa_{n}\ell+1)\bigr),
\]
of width $10\beta$ and length $\kappa_{n}\ell+1$. If $\kappa_{n}=0$, that is, if
the slab $\tau_{n-1}+1$ is retained, then $h_{n}=\tilde h_{n}$ and
$V_{i,n}$ lies on the single line of height $h_{n}$. If $\kappa_{n}\ge1$, $V_{i,n}$ is a
rectangular box made of the $\kappa_{n}$ discarded slabs
$\tau_{n-1}+1,\dots,\tau_{n}-1$ and of the lowest line of the slab $\tau_{n}$.
Therefore, the bottom and top sides of $V_{i,n}$ lie at the heights $h_{n}$ and $\tilde h_{n}$, and
coincide when $\kappa_{n}=0$. Their midpoints
\[
  z_{i,n}:=(ig,\,h_{n}),
  \qquad
  \tilde z_{i,n}:=(ig,\,\tilde h_{n}),
\]
serve as reference points. Both sides are cut into five segments
$[ig+(2m-5)\beta,\,ig+(2m-3)\beta)$ of width $2\beta$, for $0\le m\le4$. The two outer segments of
the top side carry the \emph{departure intervals}, the two inner segments of the bottom
side carry the \emph{arrival intervals}, up to a small translation towards the center,
and, when $\kappa_{n}\ge1$, the central segment carries the \emph{corridor}, which crosses
from the bottom side to the top side. These three objects are described below, but before doing so, we need to introduce the
transport boxes, which realize the edges and are inherited from
Proposition~\ref{prop:boxsuper}.

\emph{Transport boxes.} For $z\in\R^{2}$ set
\begin{align}
  \label{eq:transport}
  &B^{+}(z):=z+(-5\beta,0)+R\bigl((2\beta,0),(18\beta,18L)\bigr),\\
  &B^{-}(z):=z+(3\beta,0)+R\bigl((2\beta,0),(-18\beta,18L)\bigr), \notag
\end{align}
and the \emph{truncated transport boxes}, obtained by truncating at the length $6L$,
\begin{align}
  \label{eq:transportflat}
  &B^{+,t}(z):=z+(-5\beta,0)+R\bigl((2\beta,0),(6\beta,6L)\bigr),\\
  &B^{-,t}(z):=z+(3\beta,0)+R\bigl((2\beta,0),(-6\beta,6L)\bigr). \notag
\end{align}
Thus $B^{\pm,t}(z)$ is the part of $B^{\pm}(z)$ lying below the relative
height $6L$, and the restriction of a vertical crossing of $B^{\pm}(z)$ to those
lines is a vertical crossing of $B^{\pm,t}(z)$, so that
\begin{equation}
  \label{eq:flatincl}
  C^{\uparrow}\bigl(B^{\pm}(z)\bigr)\subseteq C^{\uparrow}\bigl(B^{\pm,t}(z)\bigr).
\end{equation}
The two edges leaving $(i,n)$ are based at $\tilde z_{i,n}$, and we write
$B^{\pm}_{i,n}:=B^{\pm}(\tilde z_{i,n})$ and likewise $B^{\pm,t}_{i,n}$. Their
bases are $2\beta$ wide and $6\beta$ apart, leaving space between them for the two arrival intervals and, when $\kappa_n\ge 1$, the corridor.

\emph{The departure intervals.}
The bases of the two transport boxes leaving $(i,n)$ are the two outer segments of its
top side,
\begin{equation}
  \label{eq:bases}
  [ig-5\beta,\,ig-3\beta)
  \qquad\text{and}\qquad
  [ig+3\beta,\,ig+5\beta),
\end{equation}
by~\eqref{eq:transport}.

\emph{The arrival intervals.}
The grid is such that a transport box does not arrive where it started relative to $ig$. Indeed, if the arrival interval from $(i-1,n-1)$ (resp. $(i+1,n-1)$) were the departure one to $(i+1,n+1)$ (resp. $(i-1,n+1)$), the transport boxes would be aligned and nothing would force an incoming path to cross an outgoing one. The drift over one slab is $12\beta$ against a column spacing $g\le10\beta$,
so the arriving transport boxes are shifted from the departure intervals of $(i,n)$ by one segment of width $2\beta$ up to the rounding
\begin{equation}
  \label{eq:Delta}
  \Delta:=10\beta-g,
  \qquad\text{so that}\qquad
  0\le\Delta<2
\end{equation}
by~\eqref{eq:g}. The \emph{arrival intervals} of $(i,n)$ are therefore the two inner segments of its bottom side, translated towards
the center by $\Delta$,
\begin{equation}
  \label{eq:arrival}
  S^{+}_{i,n}:=[ig+\Delta-3\beta,\ ig+\Delta-\beta),
  \qquad
  S^{-}_{i,n}:=[ig+\beta-\Delta,\ ig+3\beta-\Delta).
\end{equation}
The two transport boxes arriving then exceed these intervals by $6L$, which is what will let the underlying paths of
$\Ze$ meet each other. The rounding $\Delta$ is needed because $\beta=\alpha L$ is not necessarily an even integer.

\emph{The corridor.}
When $\kappa_{n}\ge1$ the two heights of the floor $n$ differ and a path arriving
at $h_{n}$ must reach $\tilde h_{n}$ before it can leave, across $\kappa_{n}$
slabs carrying no crossing estimate close to one. There is no such corridor when $\kappa_{n}=0$. Denote by
\begin{equation}
  \label{eq:corridor}
  B^{0}_{i,n}:=z_{i,n}-\bigl(\frac{3}{2}w,0\bigr)
    +R\bigl((3w,0),(0,(2\kappa_{n}+1)\,6L)\bigr)
\end{equation}
the corridor of $(i,n)$, a rectangular box of width $3w$ and length
$(2\kappa_{n}+1)\,6L=\kappa_{n}\ell+6L$, centered in the central segment of the vertex: it
starts at the height $h_{n}$, spans the $\kappa_{n}$ discarded slabs, and exceeds
$\tilde h_{n}$ by $6L$, which is necessary for the underlying paths of $\Ze$ to
meet. It occupies the columns
$[ig-3w/2,ig+3w/2)$, so that by~\eqref{eq:arrival} it is separated from
both arrival intervals, each lying strictly on one side of it, as soon as
$3w/2\le\beta-\Delta$. This holds under
\begin{equation}
  \label{eq:wsmall}
  3w+4\le2\beta,
\end{equation}
a condition on the width $w\ge2$, which is fixed just below.

Theorem~\ref{thm:DuminilCopinTassionTeixeira2018} gives crossings for the
rectangular boxes of the form $w_m\times 3m$ and $3w_m\times m$ with the width going to infinity as the
length does. We need more, the length must be unbounded while the width stays
below a fixed fraction of $\beta$, as~\eqref{eq:wsmall} requires. Stacking small rectangular boxes
of a fixed shape decouples the length from the width, at the price of a factor
$\varepsilon$ per box. Unlike in~\cite[Corollary~10]{ArcherEtAl2026}, no tilting is
needed: the critical model has zero edge speed, so a corridor is a rectangular box.

From now on, we fix $w=w_{6L}$, the width given by Theorem~\ref{thm:DuminilCopinTassionTeixeira2018} at the
length $6L$. Both $w\ge2$ and~\eqref{eq:wsmall}
are then conditions on $L$ alone, satisfied for $L$ large since
$\varepsilon(6L)^{2/5}\le w_{6L}\le(6L)^{1-\varepsilon}$ while $\beta=\alpha L$. Since $w$ is now fixed, we drop it from the notation.

To cross the corridor of width $3w$, we use what we call a \emph{ladder}. Its two rails
are the long rectangular boxes $T_{j}$, of width $w$ and length $18L$, which alternate
between the columns $[0,w)$ and $[2w,3w)$; its rungs are the wide rectangular boxes
$W_{j}$, of width $3w$ and length $6L$, which meet the two rails exactly at their
ends. Fix $N=2k+1\ge3$ and a base point $z\in\R^{2}$. For $k=1$ set
\begin{equation}
  \label{eq:lad3}
  \Lambda_{3}(z):=z+R((w,0),(0,18L)),
  \qquad
  \mathrm{Lad}\bigl(\Lambda_{3}(z)\bigr):=C^{\uparrow}\bigl(\Lambda_{3}(z)\bigr),
\end{equation}
and for $k\ge2$ set
\begin{align}
  T_{j}&:=z+(2w\,\ind\{j\ \text{odd}\},\,12Lj)+R((w,0),(0,18L)),
  &&0\le j\le k-1,
  \label{eq:Tj}\\
  W_{j}&:=z+(0,\,12L(j+1))+R((3w,0),(0,6L)),
  &&0\le j\le k-2 ,
  \label{eq:Wj}
\end{align}
\[
  \Lambda_{N}(z):=\bigcup_{j}T_{j}\cup\bigcup_{j}W_{j},
\]
\begin{equation}
  \label{eq:lad}
  \mathrm{Lad}\bigl(\Lambda_{N}(z)\bigr):=
  \bigcap_{0\le j\le k-1}C^{\uparrow}(T_{j})
  \ \cap\
  \bigcap_{\substack{0\le j\le k-2\\ j\ \text{even}}}C^{\rightarrow}(W_{j})
  \ \cap\
  \bigcap_{\substack{0\le j\le k-2\\ j\ \text{odd}}}C^{\leftarrow}(W_{j}),
\end{equation}
the event that the ladder is climbed. The length of a ladder is $6LN$ and its width is $3w$, independent of $N$, so we can define, for $\kappa_n\ge 1$, the ladder of $(i,n)$ as the one filling its corridor,
\begin{equation}
  \label{eq:ladder}
  \Lambda_{i,n}:=\Lambda_{2\kappa_{n}+1}
    \Bigl(z_{i,n}-\bigl(\frac{3}{2}w,0\bigr)\Bigr)\ \subseteq\ B^{0}_{i,n}.
\end{equation}

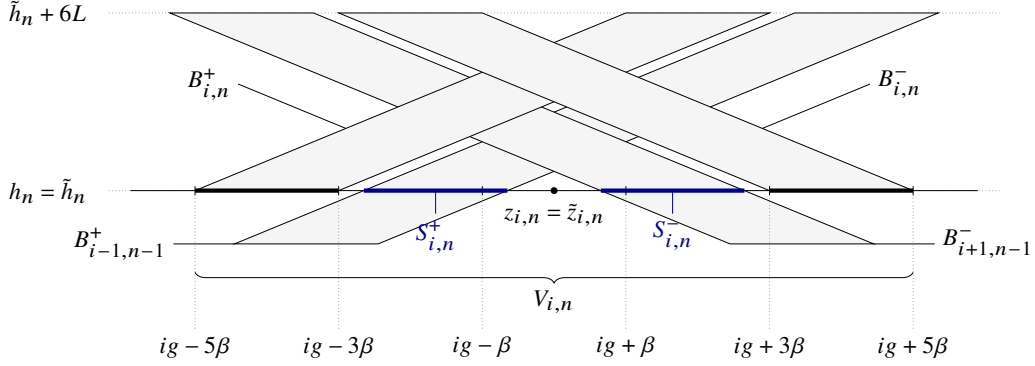
\begin{figure}[tb]
\centering
\begin{tikzpicture}[x=.95cm,y=2.35cm]
  \def\D{.35}
  \foreach \v/\lab in {0/{$h_{n}=\tilde h_{n}$}, 1/{$\tilde h_{n}+6L$}}
    {\draw[densely dotted,line width=.25pt,gray] (-6.2,\v) -- (6.2,\v);
     \node[left,font=\small] at (-6.4,\v) {\lab};}
  \fill[gray!8] (-4.8+\D,-.3) -- (-2.8+\D,-.3) -- (5+\D,1) -- (3+\D,1) -- cycle;
  \draw[box] (-4.8+\D,-.3) -- (3+\D,1) -- (5+\D,1) -- (-2.8+\D,-.3) -- cycle;
  \fill[gray!8] (2.8-\D,-.3) -- (4.8-\D,-.3) -- (-3-\D,1) -- (-5-\D,1) -- cycle;
  \draw[box] (4.8-\D,-.3) -- (-3-\D,1) -- (-5-\D,1) -- (2.8-\D,-.3) -- cycle;
  \fill[gray!8] (-5,0) -- (-3,0) -- (3,1) -- (1,1) -- cycle;
  \draw[box] (-5,0) -- (1,1) -- (3,1) -- (-3,0);
  \fill[gray!8] (3,0) -- (5,0) -- (-1,1) -- (-3,1) -- cycle;
  \draw[box] (5,0) -- (-1,1) -- (-3,1) -- (3,0);
  \draw[line width=.3pt] (-5.9,0) -- (5.9,0);
  \foreach \x in {-5,-3,-1,1,3,5} \draw[line width=.3pt] (\x,-.025) -- (\x,.025);
  \draw[line width=1.6pt] (-5,0) -- (-3,0);
  \draw[line width=1.6pt] (3,0) -- (5,0);
  \draw[line width=1.6pt,blue!60!black] (-3+\D,0) -- (-1+\D,0);
  \draw[line width=1.6pt,blue!60!black] (1-\D,0) -- (3-\D,0);
  \fill (0,0) circle (1.3pt);
  \node[below,font=\small] at (0,-.03) {$z_{i,n}=\tilde z_{i,n}$};
  \draw[line width=.25pt] (-2.9,.35) -- (-4.4,.60);
  \node[left,font=\small] at (-4.4,.60) {$B^{+}_{i,n}$};
  \draw[line width=.25pt] (2.9,.35) -- (4.4,.60);
  \node[right,font=\small] at (4.4,.60) {$B^{-}_{i,n}$};
  \draw[line width=.25pt] (-4.45,-.3) -- (-5.3,-.3);
  \node[left,font=\small] at (-5.3,-.3) {$B^{+}_{i-1,n-1}$};
  \draw[line width=.25pt] (4.45,-.3) -- (5.3,-.3);
  \node[right,font=\small] at (5.3,-.3) {$B^{-}_{i+1,n-1}$};
  \draw[decorate,decoration={brace,mirror,amplitude=4pt},line width=.3pt]
    (-5,-.46) -- (5,-.46) node[midway,below,yshift=-4pt,font=\small] {$V_{i,n}$};
  \draw[line width=.25pt,blue!60!black] (-2+\D,0) -- (-2+\D,-.13);
  \node[below,font=\small,blue!60!black] at (-2+\D,-.13) {$S^{+}_{i,n}$};
  \draw[line width=.25pt,blue!60!black] (2-\D,0) -- (2-\D,-.13);
  \node[below,font=\small,blue!60!black] at (2-\D,-.13) {$S^{-}_{i,n}$};
  \foreach \x/\lab in {-5/{$ig-5\beta$}, -3/{$ig-3\beta$}, -1/{$ig-\beta$},
                       1/{$ig+\beta$}, 3/{$ig+3\beta$}, 5/{$ig+5\beta$}}
    {\draw[densely dotted,line width=.25pt,gray] (\x,-.025) -- (\x,-.78);
     \node[below,font=\small] at (\x,-.78) {\lab};}
\end{tikzpicture}
\caption{The vertex $(i,n)$ when $\kappa_{n}=0$.}
\label{fig:goodvertex}
\end{figure}
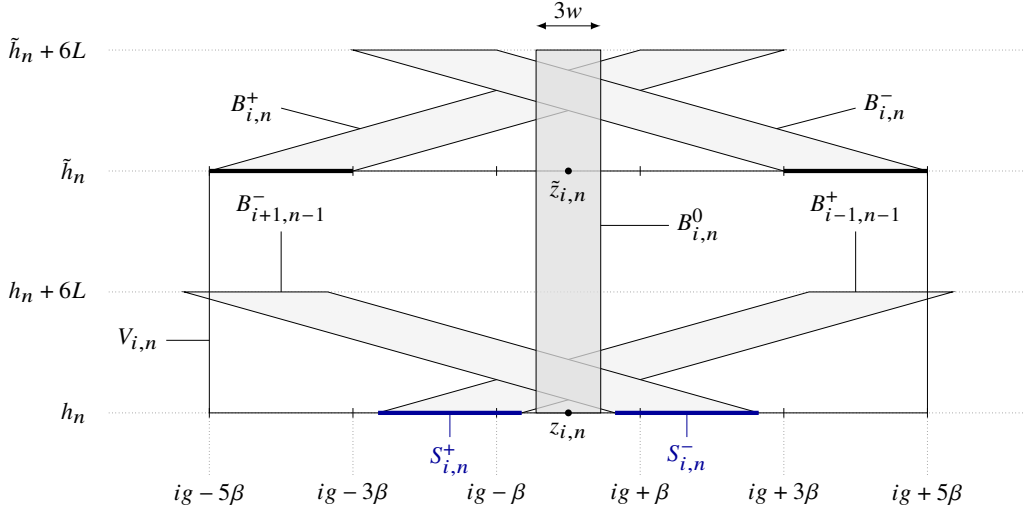
\begin{figure}[tb]
\centering
\begin{tikzpicture}[x=.95cm,y=1.6cm]
  \def\D{.35}
  \def\hw{.45}
  \foreach \v/\lab in {0/{$h_{n}$}, 1/{$h_{n}+6L$}, 2/{$\tilde h_{n}$}, 3/{$\tilde h_{n}+6L$}}
    {\draw[densely dotted,line width=.25pt,gray] (-6.4,\v) -- (6.4,\v);
     \node[left,font=\small] at (-6.6,\v) {\lab};}
  \fill[gray!8] (-3+\D,0) -- (-1+\D,0) -- (5+\D,1) -- (3+\D,1) -- cycle;
  \draw[box] (-3+\D,0) -- (3+\D,1) -- (5+\D,1) -- (-1+\D,0);
  \fill[gray!8] (1-\D,0) -- (3-\D,0) -- (-3-\D,1) -- (-5-\D,1) -- cycle;
  \draw[box] (1-\D,0) -- (-5-\D,1) -- (-3-\D,1) -- (3-\D,0);
  \fill[gray!8] (-5,2) -- (-3,2) -- (3,3) -- (1,3) -- cycle;
  \draw[box] (-5,2) -- (1,3) -- (3,3) -- (-3,2);
  \fill[gray!8] (3,2) -- (5,2) -- (-1,3) -- (-3,3) -- cycle;
  \draw[box] (5,2) -- (-1,3) -- (-3,3) -- (3,2);
  \draw[line width=.3pt] (-5,0) rectangle (5,2);
  \fill[gray!25,fill opacity=.75] (-\hw,0) rectangle (\hw,3);
  \draw[line width=.35pt] (-\hw,0) rectangle (\hw,3);
  \foreach \x in {-5,-3,-1,1,3,5} {\draw[line width=.3pt] (\x,-.04) -- (\x,.04);
                                   \draw[line width=.3pt] (\x,1.96) -- (\x,2.04);}
  \draw[line width=1.6pt,blue!60!black] (-3+\D,0) -- (-1+\D,0);
  \draw[line width=1.6pt,blue!60!black] (1-\D,0) -- (3-\D,0);
  \draw[line width=1.6pt] (-5,2) -- (-3,2);
  \draw[line width=1.6pt] (3,2) -- (5,2);
  \fill (0,0) circle (1.3pt); \fill (0,2) circle (1.3pt);
  \node[below,font=\small] at (0,-.02) {$z_{i,n}$};
  \node[below,font=\small] at (0,1.98) {$\tilde z_{i,n}$};
  \draw[line width=.25pt,blue!60!black] (-1.65,0) -- (-1.65,-.2);
  \node[below,font=\small,blue!60!black] at (-1.65,-.2) {$S^{+}_{i,n}$};
  \draw[line width=.25pt,blue!60!black] (1.65,0) -- (1.65,-.2);
  \node[below,font=\small,blue!60!black] at (1.65,-.2) {$S^{-}_{i,n}$};
  \draw[line width=.25pt] (-2.9,2.35) -- (-4.0,2.52);
  \node[left,font=\small] at (-4.0,2.52) {$B^{+}_{i,n}$};
  \draw[line width=.25pt] (2.9,2.35) -- (4.0,2.52);
  \node[right,font=\small] at (4.0,2.52) {$B^{-}_{i,n}$};
  \draw[line width=.25pt] (\hw,1.55) -- (1.4,1.55);
  \node[right,font=\small] at (1.4,1.55) {$B^{0}_{i,n}$};
  \draw[line width=.25pt] (-5,.6) -- (-5.6,.6);
  \node[left,font=\small] at (-5.6,.6) {$V_{i,n}$};
  \draw[line width=.25pt] (4.0,1) -- (4.0,1.5);
  \node[above,font=\small] at (4.0,1.5) {$B^{+}_{i-1,n-1}$};
  \draw[line width=.25pt] (-4.0,1) -- (-4.0,1.5);
  \node[above,font=\small] at (-4.0,1.5) {$B^{-}_{i+1,n-1}$};
  \draw[cote] (-\hw,3.2) -- (\hw,3.2) node[midway,above,font=\small] {$3w$};
  \foreach \x/\lab in {-5/{$ig-5\beta$}, -3/{$ig-3\beta$}, -1/{$ig-\beta$},
                       1/{$ig+\beta$}, 3/{$ig+3\beta$}, 5/{$ig+5\beta$}}
    {\draw[densely dotted,line width=.25pt,gray] (\x,-.04) -- (\x,-.52);
     \node[below,font=\small] at (\x,-.52) {\lab};}
\end{tikzpicture}
\caption{The vertex $(i,n)$ when $\kappa_{n}\ge 1$, here $\kappa_{n}=1$.}
\label{fig:badvertex}
\end{figure}

\begin{definition}[Open vertices and edges]
\label{def:renorm}
Let $(i,n)\in\mathbb{V}$. The vertex $(i,n)$ is \emph{open} if $\kappa_{n}=0$,
or if $\kappa_{n}\ge1$ and $\mathrm{Lad}(\Lambda_{i,n})$ occurs. Write
$Y_{i,n}:=\ind\{(i,n) \text{ is open}\}$. The edge
$\bigl((i,n),(i+1,n+1)\bigr)$ is \emph{open} if
\begin{equation*}
  Z^{+}_{i,n}:=
  \ind\Bigl\{C^{\uparrow}\bigl(B^{+}_{i,n}\bigr)\cap
             C^{\uparrow}\bigl(B^{-,t}_{i,n}\bigr)\Bigr\}=1,
\end{equation*}
and the edge $\bigl((i,n),(i-1,n+1)\bigr)$ is \emph{open} if
\begin{equation*}
  Z^{-}_{i,n}:=
  \ind\Bigl\{C^{\uparrow}\bigl(B^{-}_{i,n}\bigr)\cap
             C^{\uparrow}\bigl(B^{+,t}_{i,n}\bigr)\Bigr\}=1.
\end{equation*}
An open path in $\mathbb{V}$ is a sequence
$(i_{0},n_{0}),(i_{1},n_{1}),\dots$ of open vertices joined by open edges.
\end{definition}

An edge requires two crossings: the transport box in its own direction, which carries the path to the next vertex, and the truncated transport box in the other direction. The second is what lets the underlying paths in the edges going in the same direction connect when $\kappa_{n}=0$. When
$\kappa_{n}\ge1$ the corridor already connects the incoming paths to the outgoing ones. This is the content of
Proposition~\ref{prop:reconstruction}.

\subsubsection{From the renormalized lattice to \texorpdfstring{$\Ze$}{Ze}}
\label{sec:renormZe}

Every result in this section holds for an arbitrary sequence
$(X_{k})_{k\ge0}$: no probability is involved. The main result of this section is the following.

\begin{proposition}[Reconstruction of an infinite open path]
\label{prop:reconstruction}
For every sequence $(X_{k})_{k\ge0}$, if there is an infinite open path in
$\mathbb{V}$ starting at the floor $0$, then there is an infinite open path in $\Ze$.
\end{proposition}

We first state a lemma which says that a climbed
ladder does contain an open path on $\Ze$ that crosses the associated corridor.

\begin{lemma}
\label{lem:ladderpath}
Let $N=2k+1\ge3$ be odd and let $z\in\R\times[0,\infty)$. On the event
$\mathrm{Lad}(\Lambda_{N}(z))$ there is an open path $\Gamma$ contained in
$\Lambda_{N}(z)$, from a site at its lowest height to a site at its highest height,
these being $\lceil z_{2}\rceil$ and $\lceil z_{2}\rceil+6LN-1$. In particular
$C^{\uparrow}(z+R((3w,0),(0,6LN)))$ occurs.
\end{lemma}

\begin{proof}
The construction is purely geometric, so it holds at any position and we take
$z=0$. For $k=1$ there is nothing to prove, so assume $k\ge2$.

On the event $\mathrm{Lad}(\Lambda_{N}(0))$, choose once and for all, by an
arbitrary procedure, a crossing path $\gamma_{j}$ of each $T_{j}$,
$0\le j\le k-1$, and a crossing path $\sigma_{j}$ of each $W_{j}$,
$0\le j\le k-2$. Fix $j\le k-2$, and denote by $a$ and $b$ the heights of the two
ends of $\sigma_{j}$. Both $\gamma_{j}$ and $\gamma_{j+1}$ are defined at
every height of $W_{j}$, these being the top $6L$ lines of $T_{j}$ and the bottom
$6L$ lines of $T_{j+1}$, thus they are in particular defined on $[a,b]$.

Suppose $j$ is even, so that $T_{j}$ occupies the columns $[0,w)$ and $T_{j+1}$
the columns $[2w,3w)$, while $W_{j}$ occupies $[0,3w)$. Being an open path of
$\Ze$, $\sigma_j$ moves by $\pm1$ at each height, so it takes every integer value
between $\sigma_j(a) \in \{0,1\}$ and $\sigma_j(b) \in \{3w-2,3w-1\}$, depending
on the parity of the height. In particular the heights
\[
    s_{1}:=\min\{h\in[a,b]:\sigma_j(h)=w\},
  \qquad
  s_{2}:=\min\{h\in[a,b]:\sigma_j(h)=2w\}
\]
are well defined, and $s_{1}<s_{2}$ since $\sigma_j$ reaches $2w$ only after
reaching $w$. By construction, the rectangular boxes $W_{j}$ and $T_{j}$ have the same
leftmost column and the rectangular boxes $W_{j}$ and $T_{j+1}$ the same rightmost one,
thus
\[
  \sigma_j(a)\le\gamma_{j}(a),
  \qquad
    \sigma_j(s_{1})=w\ge\gamma_{j}(s_{1}),
\]
\[
  \sigma_j(s_{2})=2w\le\gamma_{j+1}(s_{2}),
  \qquad
  \sigma_j(b)\ge\gamma_{j+1}(b).
\]
Lemma~\ref{lem:planar}(i), applied on $[a,s_{1}]$ to $\sigma_j$ and
$\gamma_{j}$, then on $[s_{2},b]$ to $\sigma_j$ and $\gamma_{j+1}$, gives
heights $s\in[a,s_{1}]$ and $s'\in[s_{2},b]$ with
$\sigma_j(s)=\gamma_{j}(s)$ and $\sigma_j(s')=\gamma_{j+1}(s')$. As
$s\le s_{1}<s_{2}\le s'$, following $\gamma_{j}$ up to the height
$s$, then $\sigma_j$ up to the height $s'$, then $\gamma_{j+1}$, gives an
open path contained in $T_{j}\cup W_{j}\cup T_{j+1}$ which crosses it from the
bottom line of $T_{j}$ to the top line of $T_{j+1}$. If $j$ is odd, the same
argument applies in the opposite direction. Iterating on $0 \le j \le k-2$ gives
$\Gamma$.

By~\eqref{eq:Tj} and~\eqref{eq:Wj} the rails occupy the columns $[0,w)$ or
$[2w,3w)$ and the rungs the columns $[0,3w)$, so that
$\Lambda_{N}(z)\subseteq z+R((3w,0),(0,6LN))$, both sets having the same lowest
and highest lines. The last assertion follows.
\end{proof}

\begin{figure}[tb]
\centering
\begin{tikzpicture}[x=.40cm,y=.21cm]
    \draw[box,fill=gray!6] (-8.75,-6) -- (3.25,6) -- (5.25,6) -- (-6.75,-6) -- cycle;
  \draw[box,fill=gray!6] (-5,0) -- (13,18) -- (15,18) -- (-3,0) -- cycle;
  \draw[box,fill=gray!6] (3,0) -- (-3,6) -- (-1,6) -- (5,0) -- cycle;
  \foreach \v/\lab in {0/{$h_{n}=\tilde h_{n}$}, 6/{$h_{n}+6L$}, 12/{$h_{n+1}$}, 18/{$h_{n}+18L$}}
    {\draw[densely dotted,line width=.25pt,gray] (-9.6,\v) -- (15.4,\v);
     \node[left,font=\small] at (-9.8,\v) {\lab};}
    \draw[opath] plot coordinates {
      (-7.802,-6.000) (-7.137,-5.400) (-6.070,-4.800) (-5.508,-4.200)
      (-4.899,-3.600) (-4.203,-3.000) (-3.950,-2.400) (-3.337,-1.800)
      (-2.594,-1.200) (-1.672,-0.600) (-1.518,-0.000) (-1.134,0.600)
      (-0.985,1.200) (-0.044,1.800) (0.769,2.400) (1.000,3.000) (1.995,3.600)
      (3.106,4.200) (3.750,4.800) (4.350,5.400) (4.573,6.000)};
  \draw[opath] plot coordinates {
      (-3.496,0.000) (-3.292,0.600) (-3.216,1.200) (-2.067,1.800)
      (-1.814,2.400) (-1.632,3.000) (-0.865,3.600) (0.000,4.000) (0.594,4.800)
      (0.898,5.400) (2.004,6.000) (2.405,6.600) (3.117,7.200) (4.192,7.800)
      (4.996,8.400) (5.700,9.000) (6.300,9.600) (6.403,10.200) (7.423,10.800)
      (8.100,11.400) (8.492,12.000) (8.751,12.600) (9.736,13.200)
      (10.427,13.800) (11.100,14.400) (11.700,15.000) (12.300,15.600)
      (12.900,16.200) (13.500,16.800) (14.100,17.400) (14.434,18.000)};
  \draw[opath] plot coordinates {
      (4.077,0.000) (3.623,0.600) (3.372,1.200) (2.527,1.800) (2.094,2.400)
      (1.000,3.000) (1.100,3.600) (0.000,4.000) (-0.278,4.800) (-0.749,5.400)
      (-1.345,6.000)};
  \node[meet] at (1,3) {}; \node[meet] at (0,4) {};
    \draw[line width=.25pt] (-5.25,-4.5) -- (-3.0,-5.7);
  \node[right,font=\small] at (-3.0,-5.7) {$B^{+}_{i-1,n-1}$};
  \draw[line width=.25pt] (10,15) -- (7.2,16.6);
  \node[left,font=\small] at (7.2,16.6) {$B^{+}_{i,n}$};
  \draw[line width=.25pt] (-2.4,5.4) -- (-5.0,6.6);
  \node[left,font=\small] at (-5.0,6.6) {$B^{-,t}_{i,n}$};
  \fill (0,0) circle (1.3pt); \fill (10,12) circle (1.3pt);
  \node[below,font=\small] at (0,-.75) {$z_{i,n}=\tilde z_{i,n}$};
  \node[below,font=\small] at (10,11.4) {$z_{i+1,n+1}$};
\end{tikzpicture}
\caption{A path crossing the vertex $(i,n)$ with $\kappa_{n}=0$, coming from $(i-1,n-1)$ and going to $(i+1,n+1)$.}
\label{fig:pathgoodvertex}
\end{figure}
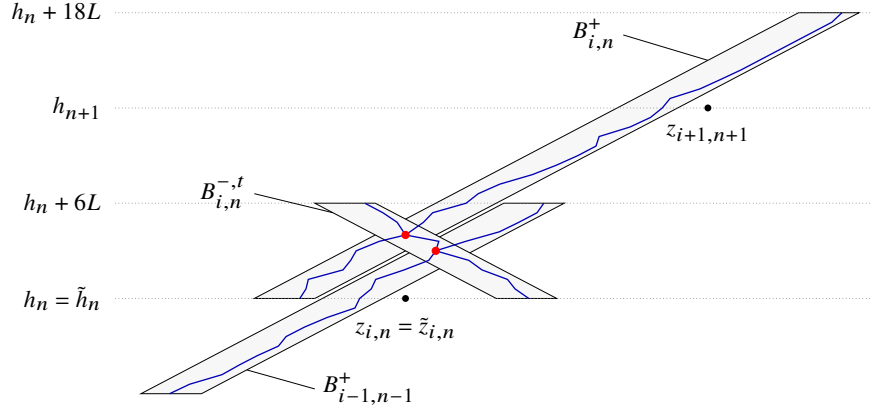

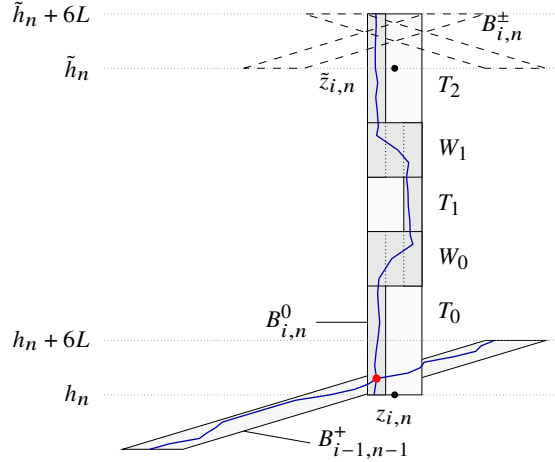
\begin{figure}[tb]
\centering
\begin{tikzpicture}[x=.40cm,y=.120cm]
  \draw[box,fill=gray!6] (-9,-6) -- (3,6) -- (5,6) -- (-7,-6) -- cycle;
  \draw[box,fill=gray!4] (-.9,0) rectangle (.9,42);
  \draw[ladder,fill=gray!18] (-.9,0) rectangle (-.3,18);
  \draw[ladder,fill=gray!18] (.3,12) rectangle (.9,30);
  \draw[ladder,fill=gray!18] (-.9,24) rectangle (-.3,42);
  \draw[ladder,fill=gray!18] (-.9,12) rectangle (.9,18);
  \draw[ladder,fill=gray!18] (-.9,24) rectangle (.9,30);
  \draw[densely dotted,line width=.35pt] (-.3,12) -- (-.3,18);
  \draw[densely dotted,line width=.35pt] (.3,12) -- (.3,18);
  \draw[densely dotted,line width=.35pt] (-.3,24) -- (-.3,30);
  \draw[densely dotted,line width=.35pt] (.3,24) -- (.3,30);
    \draw[box,dashed] (-5,36) -- (1,42) -- (3,42) -- (-3,36) -- cycle;
  \draw[box,dashed] (3,36) -- (-3,42) -- (-1,42) -- (5,36) -- cycle;
  \foreach \v/\lab in {0/{$h_{n}$}, 6/{$h_{n}+6L$}, 36/{$\tilde h_{n}$}, 42/{$\tilde h_{n}+6L$}}
    {\draw[densely dotted,line width=.25pt,gray] (-9.6,\v) -- (5.6,\v);
     \node[left,font=\small] at (-9.8,\v) {\lab};}
  \draw[opath] plot coordinates {
      (-8.101,-6.000) (-7.454,-5.400) (-6.456,-4.800) (-6.212,-4.200)
      (-5.914,-3.600) (-5.700,-3.000) (-5.100,-2.400) (-4.500,-1.800)
      (-3.900,-1.200) (-3.285,-0.600) (-2.146,-0.000) (-1.334,0.600)
      (-0.867,1.200) (-0.600,1.800) (0.439,2.400) (0.847,3.000) (0.978,3.600)
      (1.954,4.200) (2.683,4.800) (2.856,5.400) (3.300,6.000)};
  \draw[opath] plot coordinates {
      (-0.682,0.000) (-0.600,1.800) (-0.662,3.200) (-0.614,4.800)
      (-0.546,6.400) (-0.496,8.000) (-0.536,9.600) (-0.566,11.200)
      (-0.521,12.800) (-0.100,15.000) (0.250,15.800) (0.600,16.600)
      (0.514,17.600) (0.501,19.200) (0.450,20.800) (0.439,22.400)
      (0.391,24.000) (0.463,25.600) (0.100,27.000) (-0.250,27.800)
      (-0.600,28.600) (-0.540,29.600) (-0.575,31.200) (-0.627,32.800)
      (-0.562,34.400) (-0.639,36.000) (-0.644,37.600) (-0.633,39.200)
      (-0.610,40.800) (-0.622,42.000)};
  \node[meet] at (-.6,1.8) {};
  \fill (0,0) circle (1.3pt);
  \fill (0,36) circle (1.3pt);
  \node[left,font=\small] at (-1,34.3) {$\tilde z_{i,n}$};
  \draw[line width=.25pt] (-5,-4) -- (-2.8,-5.3);
  \node[right,font=\small] at (-2.8,-5.3) {$B^{+}_{i-1,n-1}$};
  \draw[line width=.25pt] (-.9,8) -- (-2.6,8);
  \node[left,font=\small] at (-2.6,8) {$B^{0}_{i,n}$};
  \node[right,font=\small] at (1.2,9) {$T_{0}$};
  \node[right,font=\small] at (1.2,15) {$W_{0}$};
  \node[right,font=\small] at (1.2,21) {$T_{1}$};
  \node[right,font=\small] at (1.2,27) {$W_{1}$};
  \node[right,font=\small] at (1.2,34) {$T_{2}$};
  \node[below,font=\small] at (0,-.80) {$z_{i,n}$};
  \node[right,font=\small] at (2.6,40.5) {$B^{\pm}_{i,n}$};
\end{tikzpicture}
\caption{A path crossing the vertex $(i,n)$ with $\kappa_{n}\ge1$, here $\kappa_{n}=3$, coming from $(i-1,n-1)$ and going to $(i+1,n+1)$.}
\label{fig:pathbadvertex}
\end{figure}

We can now prove Proposition~\ref{prop:reconstruction}. Figures~\ref{fig:pathgoodvertex} and~\ref{fig:pathbadvertex} illustrate the proof.

\begin{proof}[Proof of Proposition~\ref{prop:reconstruction}]
Let $(i_{0},0),(i_{1},1),\dots$ be an infinite open path of $\mathbb{V}$, set
$x_{n}:=i_{n}g$ and let $d_{n}\in\{\pm1\}$ be the direction of the edge $\bigl((i_{n},n),(i_{n+1},n+1)\bigr)$, so that $x_{n+1}=x_{n}+d_{n}g$.
In the following we use the abuse of notation $d_{n}\in\{-,+\}$. By Definition~\ref{def:renorm}, $Z^{d_{n}}_{i_{n},n}=1$ gives an open
path $\gamma^{d_{n}}_{n}$ crossing $B^{d_{n}}_{i_{n},n}$ and
an open path $\gamma^{-d_{n}}_{n}$ crossing
$B^{-d_{n},t}_{i_{n},n}$, both defined on
$[\tilde h_{n},\tilde h_{n}+6L-1]$ and the first up to $\tilde h_{n}+18L-1$;
and, when $\kappa_{n}\ge1$, Lemma~\ref{lem:ladderpath}, applied with $z=z_{i_n,n}-(3w/2,0)$ and $N=2\kappa_{n}+1$, gives an open path
$\Gamma_{n}\subseteq B^{0}_{i_{n},n}$ defined on
$[h_{n},\tilde h_{n}+6L-1]$. If several such paths exist, we choose them by an
arbitrary procedure.

We construct inductively, for $n\ge1$, a finite open path of $\Ze$ extending the
previous one and ending with the restriction of
$\gamma^{d_{n-1}}_{n-1}$ to $[h_{n},h_{n}+6L-1]$, with
\begin{equation}
  \label{eq:induction}
  \gamma^{d_{n-1}}_{n-1}(h_{n})\in S^{d_{n-1}}_{i_{n},n}.
\end{equation}

Their union is an infinite open path of $\Ze$. The step rests on the drift of a
transport box over $5L$ lines, namely $5\beta$ by~\eqref{eq:transport}: it is the
half-width of a vertex, so it carries a box across the whole vertex, and it is
achieved below the truncation height $6L$.

Fix $n\ge1$, assume~\eqref{eq:induction} and write $d:=d_{n-1}$,
$\gamma:=\gamma^{d}_{n-1}$. Assume first $\kappa_{n}\ge1$, so that
$\gamma$ stops at $h_{n}+6L-1<\tilde h_{n}$ and must be relayed; see
Figure~\ref{fig:pathbadvertex}. By
\eqref{eq:wsmall} and~\eqref{eq:Delta}, $3w\le2\beta-4<2\beta-2\Delta$, so
$3w/2<\beta-\Delta$ and, by~\eqref{eq:arrival} and~\eqref{eq:corridor},
\[
  d\bigl(\gamma(h_{n})-x_{n}\bigr)\le-(\beta-\Delta)<-\frac{3}{2}w
  \le d\bigl(\Gamma_{n}(h_{n})-x_{n}\bigr).
\]
By~\eqref{eq:transport} the section of $B^{d}_{i_{n-1},n-1}$ at the
relative height $12L+5L\le18L-1$ is $S^{d}_{i_{n},n}$ translated by
$5d\beta$, so
\[
  d\bigl(\gamma(h_{n}+5L)-x_{n}\bigr)\ge2\beta>\frac{3}{2}w
  \ge d\bigl(\Gamma_{n}(h_{n}+5L)-x_{n}\bigr).
\]
Both paths are defined on $[h_{n},h_{n}+5L]$, since $5L\le6L-1$ and
$\tilde h_{n}+6L-1\ge h_{n}+18L-1$, so Lemma~\ref{lem:planar}(i) gives a height
$s\le h_{n}+5L<\tilde h_{n}$ at which $\gamma$ and $\Gamma_{n}$ share a site.
Following $\gamma$ up to $s$ and $\Gamma_{n}$ afterwards extends the path up
to $\tilde h_{n}+6L-1$. If instead $\kappa_{n}=0$, then $h_{n}=\tilde h_{n}$ and there is nothing to
relay; see Figure~\ref{fig:pathgoodvertex}. In both cases, write $\pi$ for the portion lying at
the heights $\ge\tilde h_{n}$: it is defined on
$[\tilde h_{n},\tilde h_{n}+5L]$ and
$\pi(\tilde h_{n})\in[x_{n}-3\beta,x_{n}+3\beta)$, being in
$(x_{n}-\beta,x_{n}+\beta)$ by~\eqref{eq:corridor} in the first case and in
$S^{d}_{i_{n},n}$ in the second. Hence, by~\eqref{eq:bases},
\[
  \gamma^{+}_{n}(\tilde h_{n})<x_{n}-3\beta\le\pi(\tilde h_{n})<x_{n}+3\beta
  \le\gamma^{-}_{n}(\tilde h_{n}),
\]
while at the height $\tilde h_{n}+5L$, where both $\gamma^{\pm}_{n}$ are defined
since $5L\le6L-1$, the sections of $B^{\pm}_{i_{n},n}$ have drifted by
$5\beta$, so that
\[
  \gamma^{-}_{n}(\tilde h_{n}+5L)<x_{n}\le\gamma^{+}_{n}(\tilde h_{n}+5L),
\]
the strict inequality by the half-open convention. Lemma~\ref{lem:planar}(i)
gives $s'\in(\tilde h_{n},\tilde h_{n}+5L]$ at which
$\gamma^{+}_{n}$ and $\gamma^{-}_{n}$ share a site, and
Lemma~\ref{lem:planar}(ii) a height $s''\le s'$ at which $\pi$ shares
a site with one of them. If $\pi$ meets $\gamma^{d_n}_{n}$ at $s''$, we follow $\pi$ up to $s''$
and then $\gamma^{d_n}_{n}$. Otherwise it meets $\gamma^{-d_n}_{n}$, and we
follow $\pi$ up to $s''$, then $\gamma^{-d_n}_{n}$ up to $s'$, then
$\gamma^{d_n}_{n}$. That transport box is not truncated, so the resulting path is defined up to $\tilde h_{n}+18L-1=h_{n+1}+6L-1$, and at
$h_{n+1}=\tilde h_{n}+12L$ it sits in $S^{d_n}_{i_{n+1},n+1}$ by
\eqref{eq:arrival}. This is~\eqref{eq:induction} at $n+1$.

There is no need to cross the vertex $(i_{0},0)$: the path
$\gamma^{d_{0}}_{0}$ is defined on $[\tilde h_{0},h_{1}+6L-1]$ and
satisfies~\eqref{eq:induction} at $n=1$ by~\eqref{eq:arrival}.
\end{proof}

\subsubsection{Percolation of the renormalized model}
\label{sec:renormperc}

\begin{proposition}[Percolation of the renormalized model]
\label{prop:renormperc}
For $L$ large enough, depending on $\delta$ only, there is a sequence
$(X_{k})_{k\ge0}$ of $\{0,1\}$-valued random variables such that the scheme
associated with $(X_{k})_{k\ge0}$ as in Section~\ref{sec:renormlattice}
satisfies
\[
  \PP\bigl(\text{there is an infinite open path of }\mathbb{V}
  \text{ starting at }(0,0)\bigr)\ >\ 0,
\]
the vertices and edges of $\mathbb{V}$ being
open as in Definition~\ref{def:renorm}.
\end{proposition}

Before proving it, we need some preliminaries. Definition~\ref{def:renorm} attaches to each vertex and to each edge an event
measurable with respect to the occupation configuration in a region of $\Ze$,
namely $D_{i,n}$ and $E^{\pm}_{i,n}$ below. We first record how these regions
overlap. Write
\[
  D_{i,n}:=B^{0}_{i,n}\ \ (\text{with }D_{i,n}:=\emptyset\text{ if }\kappa_{n}=0),
  \qquad
  E^{+}_{i,n}:=B^{+}_{i,n}\cup B^{-,t}_{i,n},
  \qquad
  E^{-}_{i,n}:=B^{-}_{i,n}\cup B^{+,t}_{i,n}.
\]
The variable $Y_{i,n}$ is a function of the occupation configuration in $D_{i,n}$, by
\eqref{eq:ladder}, and $Z^{+}_{i,n}$ (resp. $Z^{-}_{i,n}$) is a function of the occupation configuration in
$E^{+}_{i,n}$ (resp. $E^{-}_{i,n}$). Finally let
\[
  \widehat{E}^{+}_{i,n}:=E^{+}_{i,n}\cup D_{i,n}\cup D_{i+1,n+1},
  \qquad
  \widehat{E}^{-}_{i,n}:=E^{-}_{i,n}\cup D_{i,n}\cup D_{i-1,n+1}.
\]

\begin{remark}
As in Section~\ref{sec:geometry} for a box, for a region $A\subseteq\Ze$ and a height $h$ we write $A_{h}$ for the set of $x\in\Z$ such that $(x,h)\in A$, and, for a set $I$ of heights, $A_{I}:=\bigcup_{h\in I}A_{h}$, the $x$-coordinates of the
sites of $A$ lying at a height of $I$.
\end{remark}

\begin{lemma}
\label{lem:regions}
For every sequence $(X_{k})_{k\ge0}$ the following hold.
\begin{enumerate}
\item[(i)] The sets $D_{i,n}$, $(i,n)\in\mathbb{V}$, are pairwise disjoint.
\item[(ii)] $E^{+}_{i,n}$ meets $D_{j,l}$ only for $(j,l)=(i,n)$ and
  $(j,l)=(i+1,n+1)$, and $E^{-}_{i,n}$ only for $(j,l)=(i,n)$ and
  $(j,l)=(i-1,n+1)$; that is, only for the two endpoints of the corresponding
  edge.
\item[(iii)] $\widehat{E}^{\cdot}_{i,n}\cap\widehat{E}^{\cdot}_{j,l}=\emptyset$
  whenever $|n-l|\ge2$ or $|i-j|\ge4$.
\end{enumerate}
\end{lemma}

\begin{proof}
Throughout we use $10\beta-2<g\le10\beta$, which follows from~\eqref{eq:g}, and
$2\beta\ge3w+4\ge10$, which follows from~\eqref{eq:wsmall} and $w\ge2$.

\emph{Point (i).}
By~\eqref{eq:corridor}, $D_{i,n}$ is contained in
$(ig-\beta,ig+\beta)\times[h_{n},\tilde h_{n}+6L)$,
that is in the slabs $\tau_{n-1}+1,\dots,\tau_{n}$ by~\eqref{eq:heights}. At the same floor, the indices of two vertices of $\mathbb{V}$ differ by at least $2$, hence their columns by at least $2g>2\beta$, so that the rectangles above are disjoint. Two
corridors at floors $n<l$ occupy disjoint sets of slabs, since
$\tau_{n}\le\tau_{l-1}<\tau_{l-1}+1$.

\emph{Point (ii).}  Using the symmetry of our construction, we only prove the assertion for $E^{+}_{i,n}$.
The heights of $E^{+}_{i,n}$ are, by~\eqref{eq:transport} and
\eqref{eq:transportflat}, those of $[\tilde h_{n},\tilde h_{n}+18L)$, that is
those of the slab $\tau_{n}$ together with the bottom $6L$ lines of the slab
$\tau_{n}+1$, while $D_{j,l}$ occupies the slabs $\tau_{l-1}+1,\dots,\tau_{l}$,
so only $l\in\{n,n+1\}$ can occur. Since
\[
  \bigl(B^{+}_{i,n}\bigr)_{[\tilde h_{n},\tilde h_{n}+12L)}
  \subseteq[ig-5\beta,ig+9\beta),
  \qquad
  \bigl(B^{-,t}_{i,n}\bigr)_{\Z}\subseteq(ig-3\beta,ig+5\beta),
\]
the set $E^{+}_{i,n}$ meets the slab $\tau_{n}$ inside $[ig-5\beta,ig+9\beta)$.

For $l=n$ the index $j-i$ is even, and $\bigl(D_{i+2,n}\bigr)_{\Z}$ starts after
$ig+2g-\beta>ig+19\beta-4$, while $\bigl(D_{i-2,n}\bigr)_{\Z}$ ends before
$ig-2g+\beta\le ig-19\beta+4$, hence outside $[ig-5\beta,ig+9\beta)$. Larger even values of $|j-i|$ are even further away, so
only $j=i$ remains. 

For $l=n+1$, at the heights $[\tilde h_{n}+12L,\tilde h_{n}+18L)$ only $B^{+}_{i,n}$
remains, $B^{-,t}_{i,n}$ stopping at $\tilde h_{n}+6L$, and
$\bigl(B^{+}_{i,n}\bigr)_{[\tilde h_{n}+12L,\tilde h_{n}+18L)}
\subseteq[ig+7\beta,ig+15\beta)$. As $D_{j,n+1}$ lies at heights at least
$h_{n+1}=\tilde h_{n}+12L$, these are the only sites of $E^{+}_{i,n}$ it can
meet. The index $j-i$ is odd, and
$\bigl(D_{j,n+1}\bigr)_{\Z}\subseteq(jg-\beta,jg+\beta)$. For $j=i-1$ this ends
before $ig-9\beta+2$ and for $j=i+3$ it starts after $ig+29\beta-6$, both
disjoint from $[ig+7\beta,ig+15\beta)$. Larger odd values of $|j-i|$ are further away, so only $j=i+1$
remains.

\emph{Point (iii).}
$\widehat{E}^{\cdot}_{i,n}$ is contained in the slabs
$\tau_{n-1}+1,\dots,\tau_{n+1}$, so two such sets with $|n-l|\ge2$ occupy
disjoint sets of slabs. For the $x$-coordinates,
$\bigl(B^{+}_{i,n}\bigr)_{\Z}\subseteq[ig-5\beta,ig+15\beta)$ and
$\bigl(B^{-}_{i,n}\bigr)_{\Z}\subseteq(ig-15\beta,ig+5\beta)$, while the
corridors involved are $D_{i,n}$, lying in $(ig-\beta,ig+\beta)$, and, according
to the direction, $D_{i+1,n+1}$ or $D_{i-1,n+1}$, lying in $(ig-11\beta,ig+11\beta)$ since
$g\le10\beta$. Therefore
$\bigl(\widehat{E}^{\cdot}_{i,n}\bigr)_{\Z}\subseteq(ig-15\beta,ig+15\beta)$, an
interval of width $30\beta$. Since $4g>40\beta-8>30\beta$, two such sets with
$|i-j|\ge4$ are disjoint.
\end{proof}

We now fix the two probabilistic inputs. The first one is the cost of a
corridor, given by Theorem~\ref{thm:DuminilCopinTassionTeixeira2018}. This is the survival role of the vertices in the unfavorable regimes.

\begin{lemma}
\label{lem:laddercost}
Let $\varepsilon$ be the constant of
Theorem~\ref{thm:DuminilCopinTassionTeixeira2018} and let $L$ be large enough for
that theorem to apply at the length $6L$ and for $w=w_{6L}\ge2$. Then, for every
odd $N=2k+1\ge3$,
\begin{equation}
  \label{eq:ladcost}
  \Pxi\bigl(\mathrm{Lad}(\Lambda_{N}(z))\bigr)\ \ge\ \varepsilon^{\,N-1}
  \qquad\text{for every }z\in\R\times[0,\infty)\text{ and every environment }\xi .
\end{equation}
\end{lemma}

\begin{proof}
The event $\mathrm{Lad}(\Lambda_{N}(z))$ is an intersection of crossing
events, thus it is increasing in the occupation configuration. By
\eqref{eq:sandwich} we have $\eta^{\mathrm{crit}}\le\eta^{\xi}$ for every
environment $\xi$, so that, as in
Corollary~\ref{cor:DuminilCopinTassionTeixeira2018quenched},
\[
  \Pxi\bigl(\mathrm{Lad}(\Lambda_{N}(z))\bigr)
  \ \ge\ \PP_{\pc}\bigl(\mathrm{Lad}(\Lambda_{N}(z))\bigr) .
\]
It remains to bound the right-hand side; by the convention of Section~\ref{sec:geometry}, we take $z=0$.

If $k=1$, then by~\eqref{eq:lad3} we have $\mathrm{Lad}(\Lambda_{3}(0))=C^{\uparrow}(R((w,0),(0,18L)))$, and
Theorem~\ref{thm:DuminilCopinTassionTeixeira2018} gives
$\PP_{\pc}(\mathrm{Lad}(\Lambda_{3}(0)))\ge\varepsilon\ge\varepsilon^{N-1}$,
since $\varepsilon\le1$.

Assume $k\ge2$. By~\eqref{eq:lad}, $\mathrm{Lad}(\Lambda_{N}(0))$ is the
intersection of the $k$ increasing events $C^{\uparrow}(T_{j})$ and of the $k-1$
increasing events $C^{\rightarrow}(W_{j})$ or $C^{\leftarrow}(W_{j})$. These
events are not independent, the rails and the rungs overlapping at their ends,
but the Harris--FKG inequality~\cite{Harris1960,FortuinKasteleynGinibre1971}
applies to them under the
product measure $\PP_{\pc}$ and gives
\[
  \PP_{\pc}\bigl(\mathrm{Lad}(\Lambda_{N}(0))\bigr)
    \ \ge\ \prod_{0\le j\le k-1}\PP_{\pc}\bigl(C^{\uparrow}(T_{j})\bigr)
  \prod_{\substack{0\le j\le k-2\\ j\ \text{even}}}\PP_{\pc}\bigl(C^{\rightarrow}(W_{j})\bigr)
  \prod_{\substack{0\le j\le k-2\\ j\ \text{odd}}}\PP_{\pc}\bigl(C^{\leftarrow}(W_{j})\bigr).
\]
By~\eqref{eq:Tj} each $T_{j}$ is a translate of $R((w,0),(0,18L))$ and by
\eqref{eq:Wj} each $W_{j}$ is a translate of $R((3w,0),(0,6L))$, so each of these
$2k-1$ factors is at least $\varepsilon$ by
Theorem~\ref{thm:DuminilCopinTassionTeixeira2018}. Hence
$\PP_{\pc}(\mathrm{Lad}(\Lambda_{N}(0)))\ge\varepsilon^{2k-1}\ge\varepsilon^{N-1}$,
since $2k-1\le N-1$ and $\varepsilon\le1$.
\end{proof}

The second input is the cost of a transport, which is where the supercritical
estimate enters. This is the propagation role of the edges. By
Definition~\ref{def:renorm} an edge issued from $(i,n)$ requires the crossing of
one full transport box and one truncated one, and by~\eqref{eq:flatincl} the crossing of
the full box in a given direction implies that of the truncated one. Both edges
issued from $(i,n)$ are therefore open as soon as the two full boxes are crossed,
an event we denote, for $z\in\R^{2}$, by
\[
  C^{\upuparrows}(z):=C^{\uparrow}\bigl(B^{+}(z)\bigr)\cap C^{\uparrow}\bigl(B^{-}(z)\bigr).
\]
Its probability depends on the environment along the $18L$ lines above $z$, and
Proposition~\ref{prop:boxsuper} makes it close to one on average only. 

\begin{proof}[Proof of Proposition~\ref{prop:renormperc}]
Fix $p,\theta\in(0,1)$ for which
Theorem~\ref{thm:HilarioSaSanchisTeixeira2026} applies, and let $J$ be an
integer fixed below. We compare the renormalized model with the model of that
theorem through an intermediate one: the independent bond percolation on
$\mathbb{V}$ in which each edge issued from the floor $n$ is open with
probability $p^{1+J\kappa_{n}}$, that is, the model of
Theorem~\ref{thm:HilarioSaSanchisTeixeira2026} with $J\kappa_{n}$ in place of
$\zeta_{n}$. After choosing the selection $(X_{k})_{k\ge0}$, we prove two
dominations. First, conditionally on the selection $(X_{k})_{k\ge0}$, which
determines the scheme, and on the environment $\xi$, under which the occupation
variables are independent, the probability that the renormalized model has an
infinite open path starting at $(0,0)$ is at least the probability that the
intermediate model has one. Second, the intermediate model dominates the model
of Theorem~\ref{thm:HilarioSaSanchisTeixeira2026} as soon as
$J\kappa_{n}\le\zeta_{n}$ for defects $\zeta_{n}$ as in that theorem, since then
$p^{1+J\kappa_{n}}\ge p^{1+\zeta_{n}}$; that theorem then concludes.

The second domination requires the $\kappa_{n}$ to be i.i.d.\ with a geometric
tail, whereas the slabs on which the transport boxes are likely to be crossed,
called favorable below, only form a $1$-dependent sequence; the selection
therefore discards some of them, using
\cite[Theorem~0.0]{LiggettSchonmannStacey1997}. The first domination meets two
obstacles. An edge depends on its two endpoints, through the corridors, and on
the neighboring edges, through the overlaps of Lemma~\ref{lem:regions}; the
Harris--FKG inequality removes the first dependence and
\cite[Corollary~1.4]{LiggettSchonmannStacey1997} dominates the second by an
independent field. Finally, Theorem~\ref{thm:HilarioSaSanchisTeixeira2026} has no
vertex variables, so the price $\varepsilon^{2\kappa_{n}}$ of a corridor is
charged to the two edges issued from it, at the cost $\varepsilon^{4\kappa_{n}}$
each, which $J$ converts into $p^{J\kappa_{n}}$.

The probability space is assumed to be rich enough to carry the auxiliary
independent random variables constructed below.

\emph{The selection of slabs.}
Let $\varepsilon\le1/2$ be the constant of
Theorem~\ref{thm:DuminilCopinTassionTeixeira2018}, and $J$ the smallest
integer such that
\begin{equation}
  \label{eq:J}
  p^{J}\ \le\ \varepsilon^{4}.
\end{equation}

The selection retains only slabs on which the two transport boxes are crossed
with high probability given the environment. Given $\varepsilon_{0}>0$, fixed
below, we say that the slab $k$ is \emph{favorable} if
\begin{equation}
  \label{eq:favorable}
    \Pxi\bigl(C^{\upuparrows}((0,k\ell))\bigr)\ \ge\ 1-\sqrt{\varepsilon_{0}}.
\end{equation}
The indicator $\ind\{k\text{ is favorable}\}$ is a function of
$(\xi_{j})_{k\ell\le j<k\ell+18L}$, the lines met by the transport boxes, and $18L<2\ell$,
so $\ind\{k\text{ is favorable}\}$ and $\ind\{k'\text{ is favorable}\}$ involve disjoint
sets of lines as soon as $|k-k'|\ge2$: the sequence
$(\ind\{k\text{ is favorable}\})_{k\ge0}$ is $1$-dependent.

We fix $\varepsilon_{0}$ small enough, depending on $p$ and $\theta^{J}$ only,
for~\cite[Theorem~0.0]{LiggettSchonmannStacey1997} to apply with parameter
$1-\theta^{J}$ to $1$-dependent sequences of indicators each equal to $1$ with
probability at least $1-\sqrt{\varepsilon_{0}}$, and for the first domination. In particular $\varepsilon_{0}$ does not depend on $L$.

By Proposition~\ref{prop:boxsuper}, applied at the scale $18L$, with $\rho=\alpha/9$ and $\varepsilon_{0}/2$ in each of the two directions
(by $\beta=\alpha L$ and~\eqref{eq:transport}, $B^{\pm}(z)$ is a translate of
$R\bigl((\rho n,0),(\pm\alpha n,n)\bigr)$), we can fix $L$, depending on
$\delta$ only, such that
\[
    \PP\bigl(C^{\upuparrows}(z)\bigr)\ >\ 1-\varepsilon_{0}
  \qquad\text{for every }z\in\R\times[0,\infty),
\]
and large enough also for Lemma~\ref{lem:laddercost} and~\eqref{eq:wsmall}.
Then by~\eqref{eq:favorable} and Markov's inequality,
\[
    \varepsilon_{0}\ >\ \PP\bigl(C^{\upuparrows}((0,k\ell))^{c}\bigr)
  \ =\ \EE\bigl[\Pxi\bigl(C^{\upuparrows}((0,k\ell))^{c}\bigr)\bigr]
  \ \ge\ \sqrt{\varepsilon_{0}}\;\PP(k\text{ is not favorable}),
\]
so $\PP(k\text{ is favorable})\ge1-\sqrt{\varepsilon_{0}}$ for every $k$.

By the choice of $\varepsilon_{0}$, $(\ind\{k\text{ is favorable}\})_{k\ge0}$ therefore
stochastically dominates i.i.d.\ Bernoulli variables of parameter $1-\theta^{J}$, so by
\cite[Theorem~0.0, Lemma~1.0]{LiggettSchonmannStacey1997} there is a sequence
$(X_{k})_{k\ge0}$ of such variables with $X_{k}\le\ind\{k\text{ is favorable}\}$ for
every $k$; it has
infinitely many ones almost surely. We construct $(X_{k})_{k\ge0}$ from $\xi$ and additional randomness independent
of $(\xi,(U_{z})_{z\in\Ze})$, so that the conditional law of $\eta$ given
$\sigma(\xi,(X_{k})_{k\ge0})$ is $\Pxi$. Moreover, by
\eqref{eq:tau}, the runs $\kappa_{n}$ are i.i.d.\ with
$\PP(\kappa_{n}=k)=(1-\theta^{J})\theta^{Jk}$. Finally, every retained slab is
favorable and the horizontal translation by $(ig,0)$ preserves $\Pxi$, $g$
being even, so
\begin{equation}
  \label{eq:favret}
    \Pxi\bigl(C^{\upuparrows}(\tilde z_{i,n})\bigr)
  \ =\ \Pxi\bigl(C^{\upuparrows}((0,\tau_{n}\ell))\bigr)
  \ \ge\ 1-\sqrt{\varepsilon_{0}}
  \qquad\text{for every }(i,n)\in\mathbb{V}.
\end{equation}
\emph{First domination.}
The intermediate model has independent edges and no vertices. We therefore
build, conditionally on $\sigma(\xi,(X_{k})_{k\ge0})$, an independent bond percolation lying below the
renormalized model, in which each edge issued from the floor $n$ is kept with
probability at least $p^{1+J\kappa_{n}}$.

For an edge $e$ of $\mathbb{V}$, write $a(e)$ and $b(e)$ for the vertices it is
issued from and leads to, and $Z_{e}$, $E_{e}$, $\widehat{E}_{e}$ for the
corresponding $Z^{\pm}_{i,n}$, $E^{\pm}_{i,n}$, $\widehat{E}^{\pm}_{i,n}$. Set
$\mathcal{G}:=\sigma(\xi,(X_{k})_{k\ge0})\vee\sigma(Y_{v}:v\in\mathbb{V})$ and
\[
  \widetilde{Z}_{e}:=Z_{e}\ \text{ if }Y_{a(e)}Y_{b(e)}=1,
  \qquad
  \widetilde{Z}_{e}:=1\ \text{ otherwise},
\]
so that an edge no longer depends on its endpoints being open.

We check, conditionally on $\mathcal{G}$, the hypotheses of
\cite[Corollary~1.4]{LiggettSchonmannStacey1997} for $(\widetilde{Z}_{e})_{e}$. Since $\{Y_{a(e)}Y_{b(e)}=1\}$ is $\mathcal{G}$-measurable, the definition of
$\widetilde{Z}_{e}$ gives
\[
  \PP\bigl(\widetilde{Z}_{e}=1\mid\mathcal{G}\bigr)=
  \begin{cases}
    1 & \text{on }\{Y_{a(e)}Y_{b(e)}=0\},\\[2pt]
    \PP\bigl(Z_{e}=1\mid\mathcal{G}\bigr) & \text{on }\{Y_{a(e)}Y_{b(e)}=1\}.
  \end{cases}
\]
On $\{Y_{a(e)}Y_{b(e)}=1\}$,
\[
  \PP\bigl(Z_{e}=1\mid\mathcal{G}\bigr)
  =\Pxi\bigl(Z_{e}=1\mid Y_{a(e)}Y_{b(e)}=1\bigr)
  \ge\Pxi(Z_{e}=1)
  \ge\Pxi\bigl(C^{\upuparrows}(\tilde z_{a(e)})\bigr)
  \ge1-\sqrt{\varepsilon_{0}}.
\]
The equality holds because, conditionally on $\sigma(\xi,(X_{k})_{k\ge0})$,
$\eta$ has law $\Pxi$, under which the configurations on the corridors $D_{v}$,
pairwise disjoint by Lemma~\ref{lem:regions}(i), and on the sites lying in no
corridor are independent, each $Y_{v}$ depending on $D_{v}$ only by
\eqref{eq:ladder}. As $Z_{e}$ is a function of $\eta$ on $E_{e}$, which meets no
corridor other than $D_{a(e)}$ and $D_{b(e)}$ by Lemma~\ref{lem:regions}(ii),
conditioning on $\mathcal{G}$ amounts to conditioning on
$\{Y_{a(e)}Y_{b(e)}=1\}$. The first inequality holds by the Harris--FKG
inequality~\cite{Harris1960,FortuinKasteleynGinibre1971}, the events
$\{Z_{e}=1\}$ and $\{Y_{a(e)}Y_{b(e)}=1\}$ being increasing in $\eta$. The second follows from
$C^{\upuparrows}(\tilde z_{a(e)})\subseteq\{Z_{e}=1\}$, by~\eqref{eq:flatincl} and
Definition~\ref{def:renorm}. The last one is~\eqref{eq:favret}.

We now turn to the range of the dependence. Let $S_{1},S_{2}$ be two sets of edges such
that for every edge of $S_{1}$ issued from $(i,n)$ and every edge of
$S_{2}$ issued from $(j,l)$, we have $|n-l|\ge2$ or $|i-j|\ge4$. By
Lemma~\ref{lem:regions}(iii), $\bigcup_{e\in S_{1}}\widehat{E}_{e}$ and
$\bigcup_{e\in S_{2}}\widehat{E}_{e}$ are disjoint. Each $\widehat{E}_{e}$ is
made of $D_{a(e)}$, $D_{b(e)}$ and of sites lying in no corridor by
Lemma~\ref{lem:regions}(ii), and these pieces remain independent conditionally
on $\mathcal{G}$, each $Y_{v}$ depending on $D_{v}$ only. As $\widetilde{Z}_{e}$
is a function of $Y_{a(e)}$, $Y_{b(e)}$, which are $\mathcal{G}$-measurable, and
of $\eta$ on $E_{e}\subseteq\widehat{E}_{e}$, the families
$(\widetilde{Z}_{e})_{e\in S_{1}}$ and $(\widetilde{Z}_{e})_{e\in S_{2}}$ are
independent conditionally on $\mathcal{G}$. Hence, conditionally on $\mathcal{G}$, $(\widetilde{Z}_{e})_{e}$ has a
dependence graph of bounded degree, which depends neither on $\xi$ nor on
$(X_{k})_{k\ge0}$. Therefore, by the choice of $\varepsilon_{0}$ and
\cite[Corollary~1.4, Lemma~1.0]{LiggettSchonmannStacey1997}, there are i.i.d.\ Bernoulli
variables $(\psi_{e})_{e}$ of mean $p$ and independent of
$\mathcal{G}$ such that $\psi_{e}\le\widetilde{Z}_{e}$.

The price of a corridor must now be shared between the two edges issued from it.
By~\eqref{eq:ladder}, $Y_{i,n}$ is a function of $\eta$ on $D_{i,n}$, so the
$Y_{i,n}$ are independent conditionally on $\sigma(\xi,(X_{k})_{k\ge0})$ by
Lemma~\ref{lem:regions}(i), and by~\eqref{eq:ladcost} with $N=2\kappa_{n}+1$,
\[
  \PP\bigl(Y_{i,n}=1\mid\sigma(\xi,(X_{k})_{k\ge0})\bigr)\ \ge\ \varepsilon^{2\kappa_{n}}
  \qquad\text{if }\kappa_{n}\ge1,
  \qquad
  Y_{i,n}=1\ \text{if }\kappa_{n}=0 .
\]
If $\kappa\ge1$, then
$1-(1-\varepsilon^{4\kappa})^{2}\le2\varepsilon^{4\kappa}\le\varepsilon^{2\kappa}$
since $\varepsilon\le1/2$, and the inequality is trivial if $\kappa=0$. Hence, conditionally on $\sigma(\xi,(X_{k})_{k\ge0})$, one can construct,
independently of $(\psi_{e})_{e}$ and for each $v$ at floor $n$, two
independent Bernoulli variables of parameter $\varepsilon^{4\kappa_{n}}$, one
for each edge issued from $v$, bounded above by $Y_{v}$. We call these
variables $(\chi_{e})_{e}$; they are independent conditionally on
$\sigma(\xi,(X_{k})_{k\ge0})$, as the $Y_{v}$ are.

It remains to assemble. Let $\mathcal{O}$ be the event that there is an infinite
path of $\mathbb{V}$ starting at $(0,0)$ whose edges satisfy $Y_{a(e)}=1$ and
$\widetilde{Z}_{e}=1$. Every vertex of such a path is the starting point of an
edge of the path, so $Y_v=1$ at each of its vertices, hence $Z_{e}=\widetilde{Z}_{e}=1$ at each of
its edges. Therefore, the path is open in the sense of Definition~\ref{def:renorm}. Since
$\chi_{e}\le Y_{a(e)}$ and $\psi_{e}\le\widetilde{Z}_{e}$, an infinite path
starting at $(0,0)$ whose edges satisfy $\chi_{e}\psi_{e}=1$ realizes
$\mathcal{O}$. Conditionally on $\sigma(\xi,(X_{k})_{k\ge0})$, the variables
$\chi_{e}\psi_{e}$ are independent Bernoulli of parameter
$p\,\varepsilon^{4\kappa_{n}}\ge p^{1+J\kappa_{n}}$ by~\eqref{eq:J}, $n$ being the
floor of $a(e)$. Hence $\PP(\mathcal{O}\mid\sigma(\xi,(X_{k})_{k\ge0}))$ is at
least the probability that the intermediate model has an infinite open path
starting at $(0,0)$.

\emph{Second domination.}
Since the $\kappa_{n}$ are i.i.d.\ with $\PP(\kappa_{n}=k)=(1-\theta^{J})\theta^{Jk}$,
\[
  \PP(J\kappa_{n}\ge j)\ =\ \theta^{J\lceil j/J\rceil}\ \le\ \theta^{j},
  \qquad j\ge0 ,
\]
so $(\kappa_{n})_{n\ge0}$ can be coupled
with i.i.d.\ variables
$(\zeta_{n})_{n\ge0}$ of law $(1-\theta)\theta^{k}$ so that
$J\kappa_{n}\le\zeta_{n}$. Then $p^{1+J\kappa_{n}}\ge p^{1+\zeta_{n}}$ and, the
existence of an infinite open path being an increasing event, the probability
that the intermediate model has one starting at $(0,0)$ is at least its
probability under $\mathbb{P}^{\zeta}_{p}$, which is the second domination. By
Theorem~\ref{thm:HilarioSaSanchisTeixeira2026} the latter is positive for almost
every $\zeta$, so $\PP(\mathcal{O}\mid\sigma(\xi,(X_{k})_{k\ge0}))>0$ almost surely, and
$\PP(\mathcal{O})>0$.
\end{proof}

\begin{proof}[Proof of Theorem~\ref{thm:main}]
Fix $L$ and $(X_{k})_{k\ge0}$ as given by Proposition~\ref{prop:renormperc}. On
the event that there is an infinite open path of $\mathbb{V}$ starting at
$(0,0)$, Proposition~\ref{prop:reconstruction} gives an infinite open path of
$\Ze$, so that
\[
  \PP\bigl(\text{there is an infinite open path in }\Ze\bigr)\ >\ 0.
\]
The starting point of that path is random, whereas $\Theta$ concerns the origin.
But the event that some infinite open path exists is contained in the countable
union, over $z\in\Ze$, of the events $\{C_{z}\text{ is infinite}\}$, so one of
them has positive probability. The environment is stationary and the uniform
variables are i.i.d., thus translating $z$ to the origin gives
$\Theta(\pG,\pc,\delta)=\PP(C_{0}\text{ is infinite})>0$.
\end{proof}

\begin{acks}[Acknowledgments]
This work was conducted during an internship of the first author at the Instituto de Matemática of the Universidade Federal do Rio de Janeiro, under the supervision of the second author. The authors thank Bernardo N.~B. de Lima for proposing the problem, and Ivailo Hartarsky for helpful conversations during the XXIX Brazilian School of Probability, held at the Centro Brasileiro de Pesquisas Físicas, Rio de Janeiro, on August 3--7, 2026.

AI tools were used to improve the wording of the manuscript and the formatting of the document, and to draw the figures. They were not used to devise the proof strategy or the mathematical arguments. The authors verified all AI-assisted material and take full responsibility for the content of the paper.
\end{acks}

\bibliographystyle{spr-ims-number}
\bibliography{references}

\begin{thebibliography}{13}

\bibitem{ArcherEtAl2026}
\begin{barticle}
\bauthor{\bsnm{Archer}, \binits{E.}},
\bauthor{\bsnm{Hartarsky}, \binits{I.}},
\bauthor{\bsnm{Kolesnik}, \binits{B.}},
\bauthor{\bsnm{Olesker-Taylor}, \binits{S.}},
\bauthor{\bsnm{Schapira}, \binits{B.}} and
\bauthor{\bsnm{Valesin}, \binits{D.}}
(\byear{2026}).
\batitle{{Catalan} percolation}.
\bjtitle{Probab. Theory Related Fields}
\bvolume{195}
\bfpage{413}--\blpage{449}.
\bdoi{10.1007/s00440-025-01406-4}.
\bmrnumber{5009017}
\end{barticle}
\endbibitem

\bibitem{BezuidenhoutGrimmett1990}
\begin{barticle}
\bauthor{\bsnm{Bezuidenhout}, \binits{C.}} and
\bauthor{\bsnm{Grimmett}, \binits{G.}}
(\byear{1990}).
\batitle{The critical contact process dies out}.
\bjtitle{Ann. Probab.}
\bvolume{18}
\bfpage{1462}--\blpage{1482}.
\bdoi{10.1214/aop/1176990627}.
\bmrnumber{1071804}
\end{barticle}
\endbibitem

\bibitem{LimaSidoraviciusVares2023}
\begin{barticle}
\bauthor{\bparticle{de} \bsnm{Lima}, \binits{B. N. B.}},
\bauthor{\bsnm{Sidoravicius}, \binits{V.}} and
\bauthor{\bsnm{Vares}, \binits{M. E.}}
(\byear{2023}).
\batitle{Dependent percolation on {$\mathbb{Z}^2$}}.
\bjtitle{Braz. J. Probab. Stat.}
\bvolume{37}
\bfpage{431}--\blpage{454}.
\bdoi{10.1214/23-BJPS575}.
\bmrnumber{4634238}
\end{barticle}
\endbibitem

\bibitem{LimaUngarettiVares2024}
\begin{barticle}
\bauthor{\bparticle{de} \bsnm{Lima}, \binits{B. N. B.}},
\bauthor{\bsnm{Ungaretti}, \binits{D.}} and
\bauthor{\bsnm{Vares}, \binits{M. E.}}
(\byear{2024}).
\batitle{A note on oriented percolation with inhomogeneities and strict inequalities}.
\bjtitle{Stochastic Process. Appl.}
\bvolume{175}
\bfpage{104387}.
\bdoi{10.1016/j.spa.2024.104387}.
\bmrnumber{4758364}
\end{barticle}
\endbibitem

\bibitem{DuminilCopinTassionTeixeira2018}
\begin{barticle}
\bauthor{\bsnm{Duminil-Copin}, \binits{H.}},
\bauthor{\bsnm{Tassion}, \binits{V.}} and
\bauthor{\bsnm{Teixeira}, \binits{A.}}
(\byear{2018}).
\batitle{The box-crossing property for critical two-dimensional oriented percolation}.
\bjtitle{Probab. Theory Related Fields}
\bvolume{171}
\bfpage{685}--\blpage{708}.
\bdoi{10.1007/s00440-017-0790-0}.
\bmrnumber{3827220}
\end{barticle}
\endbibitem

\bibitem{Durrett1984}
\begin{barticle}
\bauthor{\bsnm{Durrett}, \binits{R.}}
(\byear{1984}).
\batitle{Oriented percolation in two dimensions}.
\bjtitle{Ann. Probab.}
\bvolume{12}
\bfpage{999}--\blpage{1040}.
\bdoi{10.1214/aop/1176993140}.
\bmrnumber{757768}
\end{barticle}
\endbibitem

\bibitem{DurrettGriffeath1983}
\begin{barticle}
\bauthor{\bsnm{Durrett}, \binits{R.}} and
\bauthor{\bsnm{Griffeath}, \binits{D.}}
(\byear{1983}).
\batitle{Supercritical contact processes on {$\mathbb{Z}$}}.
\bjtitle{Ann. Probab.}
\bvolume{11}
\bfpage{1}--\blpage{15}.
\bdoi{10.1214/aop/1176993655}.
\bmrnumber{682796}
\end{barticle}
\endbibitem

\bibitem{FortuinKasteleynGinibre1971}
\begin{barticle}
\bauthor{\bsnm{Fortuin}, \binits{C. M.}},
\bauthor{\bsnm{Kasteleyn}, \binits{P. W.}} and
\bauthor{\bsnm{Ginibre}, \binits{J.}}
(\byear{1971}).
\batitle{Correlation inequalities on some partially ordered sets}.
\bjtitle{Comm. Math. Phys.}
\bvolume{22}
\bfpage{89}--\blpage{103}.
\bdoi{10.1007/BF01651330}.
\bmrnumber{309498}
\end{barticle}
\endbibitem

\bibitem{Harris1960}
\begin{barticle}
\bauthor{\bsnm{Harris}, \binits{T. E.}}
(\byear{1960}).
\batitle{A lower bound for the critical probability in a certain percolation process}.
\bjtitle{Proc. Cambridge Philos. Soc.}
\bvolume{56}
\bfpage{13}--\blpage{20}.
\bdoi{10.1017/S0305004100034241}.
\bmrnumber{115221}
\end{barticle}
\endbibitem

\bibitem{HilarioSaSanchisTeixeira2026}
\begin{barticle}
\bauthor{\bsnm{Hil{\'a}rio}, \binits{M. R.}},
\bauthor{\bsnm{S{\'a}}, \binits{M.}},
\bauthor{\bsnm{Sanchis}, \binits{R.}} and
\bauthor{\bsnm{Teixeira}, \binits{A.}}
(\byear{2026}).
\batitle{Percolation on stretched lattices: a static renormalization approach}.
\bjtitle{Probab. Theory Related Fields}
\bvolume{195}
\bfpage{451}--\blpage{496}.
\bdoi{10.1007/s00440-026-01470-4}.
\bmrnumber{5037662}
\end{barticle}
\endbibitem

\bibitem{KestenSidoraviciusVares2022}
\begin{barticle}
\bauthor{\bsnm{Kesten}, \binits{H.}},
\bauthor{\bsnm{Sidoravicius}, \binits{V.}} and
\bauthor{\bsnm{Vares}, \binits{M. E.}}
(\byear{2022}).
\batitle{Oriented percolation in a random environment}.
\bjtitle{Electron. J. Probab.}
\bvolume{27}
\bfpage{1}--\blpage{49}.
\bdoi{10.1214/22-EJP791}.
\bmrnumber{4442896}
\end{barticle}
\endbibitem

\bibitem{Liggett1985}
\begin{barticle}
\bauthor{\bsnm{Liggett}, \binits{T. M.}}
(\byear{1985}).
\batitle{An improved subadditive ergodic theorem}.
\bjtitle{Ann. Probab.}
\bvolume{13}
\bfpage{1279}--\blpage{1285}.
\bdoi{10.1214/aop/1176992811}.
\bmrnumber{806224}
\end{barticle}
\endbibitem

\bibitem{LiggettSchonmannStacey1997}
\begin{barticle}
\bauthor{\bsnm{Liggett}, \binits{T. M.}},
\bauthor{\bsnm{Schonmann}, \binits{R. H.}} and
\bauthor{\bsnm{Stacey}, \binits{A. M.}}
(\byear{1997}).
\batitle{Domination by product measures}.
\bjtitle{Ann. Probab.}
\bvolume{25}
\bfpage{71}--\blpage{95}.
\bdoi{10.1214/aop/1024404279}.
\bmrnumber{1428500}
\end{barticle}
\endbibitem

\end{thebibliography}

\end{document}